\documentclass{birkjour}
\usepackage{hyperref}

\newtheorem{thm}{Theorem}[section]

\newtheorem{lem}[thm]{Lemma}

\theoremstyle{definition}

\numberwithin{equation}{section}

\newcommand{\fa}{\mathfrak{a}}
\newcommand{\fb}{\mathfrak{b}}
\newcommand{\fc}{\mathfrak{c}}
\newcommand{\fd}{\mathfrak{d}}

\newcommand{\fg}{\mathfrak{g}}
\newcommand{\ff}{\mathfrak{f}}

\newcommand{\fh}{\mathfrak{h}}

\newcommand{\fr}{\mathfrak{r}}
\newcommand{\fs}{\mathfrak{s}}
\newcommand{\fv}{\mathfrak{v}}

\newcommand{\cA}{\mathcal{A}}
\newcommand{\cB}{\mathcal{B}}

\newcommand{\cE}{\mathcal{E}}
\newcommand{\cF}{\mathcal{F}}

\newcommand{\cK}{\mathcal{K}}
\newcommand{\cM}{\mathcal{M}}

\newcommand{\mC}{\mathbb{C}}
\newcommand{\mD}{\mathbb{D}}
\newcommand{\mN}{\mathbb{N}}
\newcommand{\mR}{\mathbb{R}}
\newcommand{\mT}{\mathbb{T}}
\newcommand{\mZ}{\mathbb{Z}}

\newcommand{\fA}{\mathfrak{A}}
\newcommand{\fB}{\mathfrak{B}}
\newcommand{\fS}{\mathfrak{S}}

\newcommand{\alg}{\operatorname{alg}}
\newcommand{\eps}{\varepsilon}
\newcommand{\Ind}{\operatorname{Ind}}
\newcommand{\Co}{\operatorname{Co}}
\newcommand{\Op}{\operatorname{Op}}

\begin{document}
%
%
%
%
%
%

\title[On a Weighted SIO with Shifts and SO Data]
{On a Weighted Singular Integral Operator\\
with Shifts and Slowly Oscillating Data}

\author[A. Yu. Karlovich]{Alexei Yu. Karlovich}
\address{
Centro de Matem\'atica e Aplica\c{c}\~oes (CMA) and\\
Departamento de Matem\'atica\\
Faculdade de Ci\^encias e Tecnologia\\
Universidade Nova de Lisboa\\
Quinta da Torre\\
2829--516 Caparica\\
Portugal} \email{oyk@fct.unl.pt}

\author[Yu. I. Karlovich]{Yuri I. Karlovich}
\address{%
Facultad de Ciencias \\
Universidad Aut\'onoma del Estado de Morelos\\
Av. Universidad 1001, Col. Chamilpa\\
C.P. 62209 Cuernavaca, Morelos \\
M\'exico} \email{karlovich@uaem.mx}

\author[A. B. Lebre]{Amarino B. Lebre}
\address{%
Departamento de Matem\'atica \\
Instituto Superior T\'ecnico \\
Universidade de Lisboa\\
Av. Rovisco Pais \\
1049--001 Lisboa \\
Portugal} \email{alebre@math.tecnico.ulisboa.pt}

\thanks{This work was partially supported by the Funda\c{c}\~ao para a Ci\^encia e a Tecnologia (Portuguese Foundation for Science and Technology)
through the projects PEst-OE/MAT/UI0297/2014 (Centro de Matem\'atica e Aplica\c{c}\~oes)
and PEst-OE/MAT/UI4032 /2014 (Centro de An\'alise Funcional e Aplica\c{c}\~oes).
The second author was also supported by the CONACYT Project No. 168104 (M\'exico) and by PROMEP (M\'exico) via ``Proyecto de Redes".}



\date{July 15, 2014}


\begin{abstract}
Let $\alpha,\beta$ be orientation-preserving diffeomorphism
(shifts) of $\mR_+=(0,\infty)$ onto itself with the only fixed
points $0$ and $\infty$ and $U_\alpha,U_\beta$ be the isometric shift
operators on $L^p(\mR_+)$ given by $U_\alpha f=(\alpha')^{1/p}(f\circ\alpha)$,
$U_\beta f=(\beta')^{1/p}(f\circ\beta)$, and $P_2^\pm=(I\pm S_2)/2$ where
\[
(S_2 f)(t):=\frac{1}{\pi i}\int\limits_0^\infty
\left(\frac{t}{\tau}\right)^{1/2-1/p}\frac{f(\tau)}{\tau-t}\,d\tau,
\quad t\in\mR_+,
\]
is the weighted Cauchy singular integral operator. We prove that
if $\alpha',\beta'$ and $c,d$ are continuous on $\mR_+$ and slowly
oscillating at $0$ and $\infty$, and
\[
\limsup_{t\to s}|c(t)|<1,
\quad
\limsup_{t\to s}|d(t)|<1,
\quad s\in\{0,\infty\},
\]
then the operator $(I-cU_\alpha)P_2^++(I-dU_\beta)P_2^-$
is Fredholm on $L^p(\mR_+)$ and its index is equal to zero.
Moreover, its regularizers are described.
\end{abstract}
\maketitle
\section{Introduction}
Let $\cB(X)$ be the Banach algebra of all bounded linear operators
acting on a Banach space $X$, and let $\cK(X)$ be the ideal of all
compact operators in $\cB(X)$. An operator $A\in\cB(X)$ is called
\textit{Fredholm} if its image is closed and the spaces $\ker A$
and $\ker A^*$ are finite-dimensional. In that case the number
\[
\Ind A:=\dim\ker A-\dim\ker A^*
\]
is referred to as the {\it index} of $A$ (see, e.g.,
\cite[Sections~1.11--1.12]{BS06}, \cite[Chap.~4]{GK92}).
For bounded linear operators $A$ and $B$, we will write $A\simeq B$ if
$A-B\in\cK(X)$. Recall that an  operator $B_r\in\cB(X)$ (resp. $B_l\in\cB(X)$)
is said to be a right  (resp. left) regularizer for $A$ if
\[
AB_r\simeq I \quad(\mbox{resp.}\quad B_lA\simeq I).
\]
It is well known that the operator $A$ is Fredholm on $X$ if and only if it
admits simultaneously a right and a left regularizer. Moreover, each right
regularizer differs from each left regularizer by a compact operator
(see, e.g., \cite[Chap.~4, Section 7]{GK92}). Therefore we may speak of a
regularizer $B=B_r=B_l$ of $A$ and two different regularizers of $A$ differ
from each other by a compact operator.

A bounded continuous function $f$ on $\mR_+=(0,\infty)$ is called
slowly oscillating (at $0$ and $\infty$) if for each (equivalently,
for some) $\lambda\in(0,1)$,
\[
\lim_{r\to s}\sup_{t,\tau\in[\lambda r,r]}|f(t)-f(\tau)|=0,
\quad s\in\{0,\infty\}.
\]
The set $SO(\mR_+)$ of all slowly oscillating functions forms a
$C^*$-algebra. This algebra properly contains $C(\overline{\mR}_+)$,
the $C^*$-algebra of all continuous functions on $\overline{\mR}_+
:=[0,+\infty]$. Suppose $\alpha$ is an orientation-preserving
diffeomorphism of $\mR_+$ onto itself, which has only two fixed
points $0$ and $\infty$. We say that $\alpha$ is a slowly oscillating
shift if $\log\alpha'$ is bounded and $\alpha'\in SO(\mR_+)$. The
set of all slowly oscillating shifts is denoted by $SOS(\mR_+)$.

Throughout the paper we suppose that $1<p<\infty$. It is easily
seen that if $\alpha\in SOS(\mR_+)$, then the shift operator
$W_\alpha$ defined by $W_\alpha f=f\circ\alpha$ is bounded and
invertible on all spaces $L^p(\mR_+)$ and its inverse is given
by $W_\alpha^{-1}=W_{\alpha_{-1}}$, where $\alpha_{-1}$ is the
inverse function to $\alpha$. Along with $W_\alpha$ we consider
the weighted shift operator
\[
U_\alpha:=(\alpha')^{1/p}W_\alpha
\]
being an isometric isomorphism of the Lebesgue space $L^p(\mR_+)$
onto itself. Let $S$ be the Cauchy singular integral operator
given by
\[
(Sf)(t):=\frac{1}{\pi i}\int\limits_0^\infty
\frac{f(\tau)}{\tau-t}\,d\tau,\quad t\in\mR_+,
\]
where the integral is understood in the principal value sense.
It is well known that $S$ is bounded on $L^p(\mR_+)$ for every
$p\in(1,\infty)$. Let $\cA$ be the smallest
closed subalgebra of $\cB(L^p(\mR_+))$ containing the identity
operator $I$ and the operator $S$. It is known
(see, e.g., \cite{D87}, \cite[Section~2.1.2]{HRS94},
\cite[Sections~4.2.2--4.2.3]{RSS11}, and \cite{SM86}) that $\cA$ is commutative
and for every $y\in(1,\infty)$ it contains the weighted singular integral
operator
\[
(S_y f)(t):=\frac{1}{\pi i}\int\limits_0^\infty
\left(\frac{t}{\tau}\right)^{1/y-1/p}\frac{f(\tau)}{\tau-t}\,d\tau,
\quad t\in\mR_+,
\]
and the operator with fixed singularities
\[
(R_y f)(t):=\frac{1}{\pi i}\int\limits_0^\infty
\left(\frac{t}{\tau}\right)^{1/y-1/p}\frac{f(\tau)}{\tau+t}\,d\tau,
\quad t\in\mR_+,
\]
which are understood in the principal value sense. For $y\in(1,\infty)$, put
\[
P_y^\pm:=(I\pm S_y)/2.
\]

This paper is in some sense a continuation of our papers \cite{KKL11a,KKL11b,KKL14a},
where singular integral operators with shifts were studied under the mild
assumptions that the coefficients belong to $SO(\mR_+)$ and the shifts belong
to $SOS(\mR_+)$. In \cite{KKL11a,KKL11b} we found a Fredholm criterion for the
singular integral operator
\[
N=(aI-bW_\alpha) P_p^++(cI-dW_\alpha)P_p^-
\]
with coefficients $a,b,c,d\in SO(\mR_+)$ and a shift $\alpha\in SOS(\mR_+)$.
However, a formula for the calculation of the index of the operator $N$ is still
missing. Further, in \cite{KKL14a} we proved that the operators
\[
A_{ij}=U_\alpha^i P_p^++U_\beta^j P_p^-,\quad i,j\in\mZ,
\]
with $\alpha,\beta\in SOS(\mR_+)$ are all Fredholm and their indices are
equal to zero. This result was the first step in the calculation of the index
of $N$. Here we make the next step towards the calculation of the index
of the operator $N$.

For $a\in SO(\mR)$, we will write $1\gg a$ if
\[
\limsup_{t\to s}|a(t)|<1,\quad s\in\{0,\infty\}.
\]
\begin{thm}[Main result]
\label{th:main}
Let $1<p<\infty$ and $\alpha,\beta\in SOS(\mR_+)$. Suppose $c,d\in SO(\mR_+)$ are
such that $1\gg c$ and $1\gg d$. Then the operator
\[
V:=(I-cU_\alpha)P_2^++(I-dU_\beta)P_2^-,
\]
is Fredholm on the space $L^p(\mR_+)$ and $\Ind V=0$.
\end{thm}
The paper is organized as follows. In Section~\ref{sec:Preliminaries} we collect
necessary facts about slowly oscillating functions and slowly oscillating shifts,
as well as about the invertibility of binomial functional operators $I-cU_\alpha$
with $c\in SO(\mR_+)$ and $\alpha\in SOS(\mR_+)$ under the assumption that $1\gg c$.
Further we prove that the operators in the algebra $\cA$ commute modulo compact
operators with the operators in the algebra $\mathcal{FO}_{\alpha,\beta}$ of functional
operators with shifts and slowly oscillating data. Finally, we show that the ranges
of two important continuous functions on $\mR$ do not contain the origin.
In Section~\ref{sec:Mellin-convolutions} we recall that the operators
$P_y^\pm$ and $R_y$, belonging to the algebra $\cA$ for every $y\in(1,\infty)$,
can be viewed as Mellin convolution operators and formulate two relations between
$P_y^+$, $P_y^-$, and $R_y$. Section~\ref{sec:Mellin-PDO} contains results on
the boundedness, compactness of semi-commutators, and the Fredholmness of Mellin
pseudodifferential operators with slowly oscillating symbols of limited smoothness
(symbols in the algebra $\widetilde{\cE}(\mR_+,V(\mR))$). Results of this section are
reformulations/modifications of corresponding results on Fourier pseudodifferential
operators obtained by the second author in \cite{K06} (see also \cite{K06-IWOTA,K08,K09}).
Notice that those results are further generalizations of earlier results by Rabinovich
(see \cite[Chap.~4]{RRS04} and the references therein) obtained for Mellin
pseudodifferential operators with $C^\infty$ slowly oscillating symbols.

In \cite[Lemma~4.4]{KKL14a} we proved that the operator
$U_\gamma R_y$ with $\gamma\in SOS(\mR_+)$ and $y\in(1,\infty)$ can be viewed
as a Mellin  pseudodifferential operator with a symbol in the algebra
$\widetilde{\cE}(\mR_+,V(\mR))$. In Section~\ref{sec:applications-Mellin-PDO}
we generalize that result and prove that $(I-vU_\gamma)R_y$ and $(I-vU_\gamma)^{-1}R_y$
with $y\in(1,\infty)$, $\gamma\in SOS(\mR_+)$, and $v\in SO(\mR_+)$ satisfying
$1\gg v$, can be viewed as Mellin pseudodifferential operators with symbols
in the algebra $\widetilde{\cE}(\mR_+,V(\mR))$. This is a key result
in our analysis.

Section~\ref{sec:proof} is devoted to the proof of Theorem~\ref{th:main}.
Here we follow the idea, which was already used in a simpler situation
of the operators $A_{ij}$ in \cite{KKL14a}. With the aid of results of
Section~\ref{sec:Preliminaries} and Section~\ref{sec:applications-Mellin-PDO},
we will show that for every $\mu\in[0,1]$ and $y\in(1,\infty)$,
the operators
\begin{align*}
&
[(I-\mu cU_\alpha)P_y^++(I-\mu dU_\beta)P_y^-]
\cdot
[(I-\mu cU_\alpha)^{-1}P_y^++(I-\mu dU_\beta)^{-1}P_y^-],
\\
&
[(I-\mu cU_\alpha)^{-1}P_y^++(I-\mu dU_\beta)^{-1}P_y^-]
\cdot
[(I-\mu cU_\alpha)P_y^++(I-\mu dU_\beta)P_y^-]
\end{align*}
are equal up to compact summands to the same operator similar to a
Mellin pseudodifferential operator with a symbol in the algebra
$\widetilde{\cE}(\mR_+,V(\mR))$. Moreover, the latter pseudodifferential operator
is Fredholm whenever $y=2$ in view of results of Section~\ref{sec:Mellin-PDO}.
This will show that each operator
\[
V_{\mu,2}=(I-\mu cU_\alpha)P_2^++(I-\mu dU_\beta)P_2^-
\]
is Fredholm on $L^p(\mR_+)$. Considering the homotopy $\mu\mapsto V_{\mu,2}$ for
$\mu\in[0,1]$, we see that the operator $V$ is homotopic to the identity operator.
Therefore, the index of $V$ is equal to zero. This will complete the proof of
Theorem~\ref{th:main}.

As a by-product of the proof of the main result, in Section~\ref{sec:Regularization}
we describe all regularizers of a slightly more general operator
\[
W=(I-cU_\alpha^{\eps_1})P_2^++(I-dU_\beta^{\eps_2})P_2^-,
\]
where $\eps_1,\eps_2\in\{-1,1\}$ and show that
\[
G_y W\simeq R_y
\]
for every $y\in(1,\infty)$, where $G_y$ is an operator similar to a Mellin
pseudodifferential operator with a symbol in $\widetilde{\cE}(\mR_+,V(\mR))$
with some additional properties. The latter relation for $y=2$ will play an
important role in the proof of an index formula for the operator $N$ in our
forthcoming work \cite{KKL15-progress}.
\section{Preliminaries}\label{sec:Preliminaries}
\subsection{Fundamental Property of Slowly Oscillating Functions}
For a unital commutative Banach algebra $\fA$, let $M(\mathfrak{A})$
denote its maximal ideal space. Identifying the points
$t\in\overline{\mR}_+$ with the evaluation functionals $t(f)=f(t)$
for $f\in C(\overline{\mR}_+)$, we get
$M(C(\overline{\mR}_+))=\overline{\mR}_+$. Consider the fibers
\[
M_s(SO(\mR_+))
:=
\big\{\xi\in M(SO(\mR_+)):\xi|_{C(\overline{\mR}_+)}=s\big\}
\]
of the maximal ideal space $M(SO(\mR_+))$ over the points
$s\in\{0,\infty\}$. By \cite[Proposition~2.1]{K08}, the set
\[
\Delta:=M_0(SO(\mR_+))\cup M_\infty(SO(\mR_+))
\]
coincides with $(\operatorname{clos}_{SO^*}\mR_+)\setminus\mR_+$
where $\operatorname{clos}_{SO^*}\mR_+$ is the weak-star closure
of $\mR_+$ in the dual space of $SO(\mR_+)$. Then $M(SO(\mR_+))
=\Delta\cup\mR_+$. By \cite[Lemma~2.2]{KKL11b}, the fibers
$M_s(SO(\mR_+))$ for $s\in\{0,\infty\}$ are connected compact Hausdorff
spaces. In what follows we write
\[
a(\xi):=\xi(a)
\]
for every $a\in SO(\mR_+)$ and every $\xi\in\Delta$.
\begin{lem}[{\cite[Proposition~2.2]{K08}}]
\label{le:SO-fundamental-property}
Let $\{a_k\}_{k=1}^\infty$ be a countable subset of $SO(\mR_+)$ and
$s\in\{0,\infty\}$. For each $\xi\in M_s(SO(\mR_+))$ there exists a
sequence $\{t_n\}_{n\in\mN}\subset\mR_+$ such that $t_n\to s$ as $n\to\infty$ and
\begin{equation}\label{eq:SO-fundamental-property}
a_k(\xi)=\lim_{n\to\infty}a_k(t_n)\quad\mbox{for all}\quad k\in\mN.
\end{equation}
Conversely, if $\{t_n\}_{n\in\mN}\subset\mR_+$ is a sequence such that $t_n\to s$
as $n\to\infty$, then there exists a functional $\xi\in M_s(SO(\mR_+))$
such that \eqref{eq:SO-fundamental-property} holds.
\end{lem}
\subsection{Slowly Oscillating Functions and Shifts}
Repeating literally the proof of \cite[Proposition~3.3]{KKL03}, we obtain the
following statement.
\begin{lem}\label{le:SO-nec}
Suppose $\varphi\in C^1(\mR_+)$ and put $\psi(t):=t\varphi'(t)$
for $t\in\mR_+$. If $\varphi,\,\psi\in SO(\mR_+)$, then
\[
\lim_{t\to s}\psi(t)=0
\quad\mbox{for}\quad
s\in\{0,\infty\}.
\]
\end{lem}
\begin{lem}[{\cite[Lemma~2.2]{KKL11a}}]
\label{le:exponent-shift}
An orientation-preserving shift $\alpha:\mR_+\to\mR_+$ belongs to $SOS(\mR_+)$
if and only if
\[
\alpha(t)=te^{\omega (t)},\quad t\in \mR_+,
\]
for some real-valued function $\omega\in SO(\mR_+)\cap C^1(\mR_+)$ such that
the function $t\mapsto t\omega^\prime(t)$ also belongs to $SO(\mR_+)$ and
$\inf_{t\in\mR_+}\big(1+t\omega'(t)\big)>0$.
\end{lem}
\begin{lem}[{\cite[Lemma~2.3]{KKL11a}}]
\label{le:composition}
If $c\in SO(\mR_+)$ and $\alpha\in SOS(\mR_+)$, then $c\circ\alpha$ belongs to
$SO(\mR_+)$ and
\[
\lim_{t\to s}(c(t)-c[\alpha(t)])=0
\quad\mbox{for}\quad
s\in\{0,\infty\}.
\]
\end{lem}
For an orientation-preserving diffeomorphism $\alpha:\mR_+\to\mR_+$, put
\[
\alpha_0(t):=t,
\quad
\alpha_i(t):=\alpha[\alpha_{i-1}(t)],
\quad
i\in\mZ,
\quad
t\in\mR_+.
\]
\begin{lem}[{\cite[Corollary~2.5]{KKL14a}}]
\label{le:iterations}
If $\alpha,\beta\in SOS(\mR_+)$, then $\alpha_i\circ\beta_j\in SOS(\mR_+)$
for all $i,j\in\mZ$.
\end{lem}
\begin{lem}\label{le:inverse-shift-fibers}
If $\alpha\in SOS(\mR_+)$, then
\[
\omega(t):=\log[\alpha(t)/t],
\quad
\widetilde{\omega}(t):=\log[\alpha_{-1}(t)/t],
\quad
t\in\mR_+,
\]
are slowly oscillating functions such that $\omega(\xi)=-\widetilde{\omega}(\xi)$
for all $\xi\in\Delta$.
\end{lem}
\begin{proof}
From Lemma~\ref{le:iterations} with $i=-1$ and $j=0$ it follows that
$\alpha_{-1}$ belongs to $SOS(\mR_+)$.
Then, by Lemma~\ref{le:exponent-shift}, $\omega,\widetilde{\omega}\in SO(\mR_+)$.
It is easy to see that
\[
\widetilde{\omega}(t)
=
\log\frac{\alpha_{-1}(t)}{t}
=
-\log\frac{t}{\alpha_{-1}(t)}
=
-\log\frac{\alpha[\alpha_{-1}(t)]}{\alpha_{-1}(t)}
=
-\omega[\alpha_{-1}(t)]
\]
for all $t\in\mR_+$. Hence, from Lemma~\ref{le:composition} it follows that
$\omega\circ\alpha_{-1}\in SO(\mR_+)$ and
\begin{equation}\label{eq:inverse-shift-fibers-1}
\lim_{t\to s}(\omega(t)+\widetilde{\omega}(t))=\lim_{t\to s}(\omega(t)-\omega[\alpha_{-1}(t)])=0,
\quad s\in\{0,\infty\}.
\end{equation}
Fix $s\in\{0,\infty\}$ and $\xi\in M_s(SO(\mR_+))$. By Lemma~\ref{le:SO-fundamental-property},
there is a sequence $\{t_j\}_{j\in\mN}\subset\mR_+$ such that $t_j\to s$ and
\begin{equation}\label{eq:inverse-shift-fibers-2}
\omega(\xi)=\lim_{j\to\infty}\omega(t_j),
\quad
\widetilde{\omega}(\xi)=\lim_{j\to\infty}\widetilde{\omega}(t_j).
\end{equation}
From \eqref{eq:inverse-shift-fibers-1}--\eqref{eq:inverse-shift-fibers-2} we obtain
\[
\omega(\xi)=\lim_{j\to\infty}\omega(t_j)-\lim_{j\to\infty}(\omega(t_j)+\widetilde{\omega}(t_j))
=
-\lim_{j\to\infty}\widetilde{\omega}(t_j)=-\widetilde{\omega}(\xi),
\]
which completes the proof.
\end{proof}
\subsection{Invertibility of Binomial Functional Operators}
From \cite[Theorem~1.1]{KKL11a} we immediately get the following.
\begin{lem}\label{le:FO}
Suppose $c\in SO(\mR_+)$ and $\alpha\in SOS(\mR_+)$. If $1\gg c$, then the
functional operator $I-cU_\alpha$ is invertible on the
space $L^p(\mR_+)$ and
\[
(I-cU_\alpha)^{-1}=\sum_{n=0}^\infty (cU_\alpha)^n.
\]
\end{lem}
\subsection{Compactness of Commutators of SIO's and FO's}
Let $\fB$ be a Banach algebra and $\fS$ be a subset of $\fB$. We denote by
$\alg_\fB\fS$ the smallest closed subalgebra of $\fB$ containing $\fS$.
Then
\[
\cA=\alg_{\cB(L^p(\mR_+))}\{I,S\}
\]
is the algebra of singular integral operators (SIO's). Fix $\alpha,\beta\in SOS(\mR_+)$
and consider the Banach algebra of functional operators (FO's) with shifts and
slowly oscillating data defined by
\[
\mathcal{FO}_{\alpha,\beta}:=
\alg_{\cB(L^p(\mR_+))}\{U_\alpha,U_\alpha^{-1},U_\beta,U_\beta^{-1},aI:a\in SO(\mR_+)\}.
\]
\begin{lem}\label{le:compactness-commutators}
Let $\alpha,\beta\in SOS(\mR_+)$. If $A\in\mathcal{FO}_{\alpha,\beta}$ and $B\in\cA$,
then
\[
AB-BA\in\cK(L^p(\mR_+)).
\]
\end{lem}
\begin{proof}
In view of \cite[Corollary~6.4]{KKL11a}, we have $aB-BaI\in\cK(L^p(\mR_+))$ for all $a\in SO(\mR_+)$
and all
$B\in\cA$. On the other hand, from \cite[Lemma~2.7]{KKL14a} it follows that
$U_\gamma^{\pm 1}B-BU_\gamma^{\pm 1}\in\cK(L^p(\mR_+))$ for all
$\gamma\in\{\alpha,\beta\}$ and $B\in\cA$. Hence, $AB-BA\in\cK(L^p(\mR_+))$
for each generator $A$ of $\mathcal{FO}_{\alpha,\beta}$ and
each $B\in\cA$. Thus, the same is true for all $A\in\mathcal{FO}_{\alpha,\beta}$
by a standard argument.
\end{proof}
\subsection{Ranges of Two Continuous Functions on \boldmath{$\mR$}}
Given $a\in\mC$ and $r>0$, let $\mD(a,r):=\{z\in\mC:|z-a|\le r\}$.
For $x\in\mR$, put
\begin{equation}\label{eq:p2}
p_2^+(x):=\frac{e^{2\pi x}}{e^{2\pi x}+1},
\quad
p_2^-(x):=\frac{1}{e^{2\pi x}+1}.
\end{equation}
\begin{lem}\label{le:range1}
Let $\psi,\zeta\in\mR$ and $v,w\in\mC$. If
\begin{equation}\label{eq:fx}
f(x):=(1-ve^{i\psi x})p_2^+(x)+(1-we^{i\zeta x})p_2^-(x),
\end{equation}
then $f(\mR)\subset\mD(1,r)$, where $r:=\max(|v|,|w|)$.
\end{lem}
\begin{proof}
From \eqref{eq:p2} and \eqref{eq:fx} we see that for every $x\in\mR$
the point $f(x)$ lies on the line segment connecting the points
$1-ve^{i\psi x}$ and $1-we^{i\zeta x}$. In turn, these points lie on
the concentric circles
\begin{equation}\label{eq:range1-1}
\{z\in\mC:|z-1|=|v|\},
\quad
\{z\in\mC:|z-1|=|w|\},
\end{equation}
respectively. Thus, each line segment mentioned above is contained
in the disk $\mD(1,r)=\{z\in\mC:|z-1|\le\max(|v|,|w|)\}$.
\end{proof}
\begin{lem}\label{le:range2}
Let $\psi,\zeta\in\mR$ and $v,w\in\mC$ with $|v|<1$, $|w|<1$. If
\begin{equation}\label{eq:gx}
g(x):=(1-ve^{i\psi x})^{-1}p_2^+(x)+(1-we^{i\zeta x})^{-1}p_2^-(x),
\quad x\in\mR,
\end{equation}
then $g(\mR)\subset\mD((1-r^2)^{-1},(1-r^2)^{-1}r)$, where
$r=\max(|v|,|w|)<1$.
\end{lem}
\begin{proof}
From \eqref{eq:p2} and \eqref{eq:gx} we see that for every $x\in\mR$
the point $g(x)$ lies on the line segment connecting the points
$(1-ve^{i\psi x})^{-1}$ and $(1-we^{i\zeta x})^{-1}$. In turn, these
points lie on the images of the circles given by \eqref{eq:range1-1}
under the inversion mapping $z\mapsto 1/z$. The image of the first
circle in \eqref{eq:range1-1} is the circle $\mT_v:=\{z\in\mC:|z-b|=\rho\}$
with center and radius given by
\begin{align*}
b &=[(1-|v|)^{-1}+(1+|v|)^{-1}]/2=(1-|v|^2)^{-1},\\
\rho &=[(1-|v|)^{-1}-(1+|v|)^{-1}]/2=(1-|v|^2)^{-1}|v|.
\end{align*}
Analogously, the image of the second circle in \eqref{eq:range1-1} is the
circle
\[
\mT_w:=
\left\{z\in\mC:\big|z-(1-|w|^2)^{-1}\big|=(1-|w|^2)^{-1}|w|\right\}.
\]
Let $\mD_v$ and $\mD_w$ be the closed disks whose boundaries are $\mT_v$ and
$\mT_w$, respectively. Obviously, one of these disks is contained in another
one, namely, $\mD_v\subset\mD_w$ if $|v|\le|w|$ and $\mD_w\subset\mD_v$
otherwise. Then each point $g(x)$, lying on the segment connecting the points
$(1-ve^{i\psi x})^{-1}\in\mT_v$ and $(1-we^{i\zeta x})^{-1}\in\mT_w$, belongs
to the biggest disk in $\{\mD_v,\mD_w\}$, that is, to the disk with center
$(1-r^2)^{-1}$ and radius $(1-r^2)^{-1}r$, where $r=\max(|v|,|w|)<1$.
\end{proof}
From Lemmas~\ref{le:range1} and~\ref{le:range2} it follows that the ranges
$f(\mR)$ and $g(\mR)$ do not contain the origin if $|v|<1$ and $|w|<1$.
\section{Weighted Singular Integral Operators Are Similar to Mellin Convolution Operators}
\label{sec:Mellin-convolutions}
\subsection{Mellin Convolution Operators}
Let $\cF:L^2(\mR)\to L^2(\mR)$ denote the Fourier transform,
\[
(\cF f)(x):=\int\limits_\mR f(y)e^{-ixy}dy,\quad x\in\mR,
\]
and let $\cF^{-1}:L^2(\mR)\to L^2(\mR)$ be the inverse of $\cF$. A function
$a\in L^\infty(\mR)$ is called a Fourier multiplier on $L^p(\mR)$
if the mapping
$f\mapsto \cF^{-1}a\cF f$ maps $L^2(\mR)\cap L^p(\mR)$ onto itself and extends
to a bounded operator on $L^p(\mR)$. The latter operator is then denoted by
$W^0(a)$. We let $\cM_p(\mR)$ stand for the set of all Fourier multipliers on
$L^p(\mR)$. One can show that $\cM_p(\mR)$ is a Banach algebra under the norm
\[
\|a\|_{\cM_p(\mR)}:=\|W^0(a)\|_{\cB(L^p(\mR))}.
\]

Let $d\mu(t)=dt/t$ be the (normalized) invariant measure on $\mR_+$.
Consider the Fourier transform on $L^2(\mathbb{R}_+,d\mu)$, which is
usually referred to as the Mellin transform and is defined by
\[
\cM:L^2(\mR_+,d\mu)\to L^2(\mR),
\quad
(\cM f)(x):=\int\limits_{\mR_+} f(t) t^{-ix}\,\frac{dt}{t}.
\]
It is an invertible operator, with inverse given by
\[
{\cM^{-1}}:L^2(\mR)\to L^2(\mR_{+},d\mu),
\quad
({\cM^{-1}}g)(t)= \frac{1}{2\pi}\int\limits_{\mR}
g(x)t^{ix}\,dx.
\]
Let $E$ be the isometric isomorphism
\begin{equation}\label{eq:def-E}
E:L^p(\mR_+,d\mu)\to L^p(\mR),
\quad
(Ef)(x):=f(e^x),\quad x\in\mR.
\end{equation}
Then the map $A\mapsto E^{-1}AE$ transforms the Fourier convolution
operator $W^0(a)=\cF^{-1}a\cF$ to the Mellin convolution operator
\[
\operatorname{Co}(a):=\cM^{-1}a\cM
\]
with the same symbol $a$. Hence the class of Fourier multipliers on
$L^p(\mR)$ coincides with the class of Mellin multipliers on $L^p(\mR_+,d\mu)$.
\subsection{Algebra \boldmath{$\cA$} of Singular Integral Operators}
\label{subsec:algebra-A}
Consider the isometric isomorphism
\begin{equation}\label{eq:def-Phi}
\Phi:L^p(\mR_+)\to L^p(\mR_+,d\mu),
\quad
(\Phi f)(t):=t^{1/p}f(t),\quad t\in\mR_+,
\end{equation}
The following statement is well known (see, e.g., \cite{D87},
\cite[Section~2.1.2]{HRS94}, and \cite[Sections~4.2.2--4.2.3]{RSS11}).
\begin{lem}\label{le:alg-A}
For every $y\in(1,\infty)$, the functions $s_y$ and $r_y$  given by
\[
s_y(x):=\coth[\pi(x+i/y)],
\quad
r_y(x):=1/\sinh[\pi(x+i/y)],
\quad x\in\mR,
\]
belong to $\cM_p(\mR)$, the operators $S_y$ and $R_y$ belong to the algebra $\cA$, and
\[
S_y=\Phi^{-1}\Co(s_y)\Phi,
\quad
R_y=\Phi^{-1}\Co(r_y)\Phi.
\]
\end{lem}
For $y\in(1,\infty)$ and $x\in\mR$, put
\[
p_y^\pm(x):=(1\pm s_y(x))/2.
\]
This definition is consistent with \eqref{eq:p2} because $s_2(x)=\tanh(\pi x)$ for
$x\in\mR$. In view of Lemma~\ref{le:alg-A} we have
\[
P_y^\pm=(I\pm S_y)/2=\Phi^{-1}\Co(p_y^\pm)\Phi.
\]
\begin{lem}\label{le:PR-relations}
\begin{enumerate}
\item[{\rm(a)}] For $y\in(1,\infty)$ and $x\in\mR$, we have
\[
p_y^+(x)p_y^-(x)=-\frac{(r_y(x))^2}{4},
\quad
(p_y^\pm(x))^2=p_y^\pm(x)+\frac{(r_y(x))^2}{4}.
\]

\item[{\rm(b)}]
For every $y\in\mR_+$, we have
\[
P_y^+P_y^-=P_y^-P_y^+=-\frac{R_y^2}{4},
\quad
(P_y^\pm)^2=P_y^\pm+\frac{R_y^2}{4}.
\]
\end{enumerate}
\end{lem}
\begin{proof}
Part (a) follows straightforwardly from the identity $s_y^2(x)-r_y^2(x)=1$.
Part (b) follows from part (a) and Lemma~\ref{le:alg-A}.
\end{proof}
\section{Mellin Pseudodifferential Operators and Their Symbols}
\label{sec:Mellin-PDO}
\subsection{Boundedness of Mellin Pseudodifferential Operators}
In 1991 Rabinovich \cite{R92} proposed to use Mellin pseudodifferential operators
with $C^\infty$ slowly oscillating symbols to study singular integral operators
with slowly oscillating coefficients on $L^p$ spaces. This idea was exploited
in a series of papers by Rabinovich and coauthors. A detailed history and a complete
bibliography up to 2004 can be found in \cite[Sections~4.6--4.7]{RRS04}. Further,
the second author developed in \cite{K06} a handy for our purposes theory of
Fourier pseudodifferential operators with slowly oscillating symbols of limited smoothness
(much less restrictive than in the works mentioned in \cite{RRS04}).
In this section we translate necessary results from \cite{K06} to the Mellin
setting with the aid of the transformation
\[
A\mapsto E^{-1}AE,
\]
where $A\in\cB(L^p(\mR))$ and the isometric isomorphism $E:L^p(\mR_+,d\mu)\to L^p(\mR)$
is defined by \eqref{eq:def-E}.

Let $a$ be an absolutely continuous function of finite total variation
\[
V(a):=\int\limits_\mR|a'(x)|dx
\]
on $\mR$. The set $V(\mR)$ of all absolutely continuous functions of finite
total variation on $\mR$ becomes a Banach algebra equipped with the norm
\begin{equation}\label{eq:norm-V}
\|a\|_V:=\|a\|_{L^\infty(\mR)}+V(a).
\end{equation}
Following \cite{K06,K06-IWOTA}, let $C_b(\mR_+,V(\mR))$ denote the Banach
algebra of all bounded continuous $V(\mR)$-valued functions on $\mR_+$ with
the norm
\[
\|\fa(\cdot,\cdot)\|_{C_b(\mR_+,V(\mR))}
=
\sup_{t\in\mR_+}\|\fa(t,\cdot)\|_V.
\]
As usual, let $C_0^\infty(\mR_+)$ be the set of all infinitely differentiable
functions of compact support on $\mR_+$.

The following boundedness result for Mellin pseudodifferential operators
follows from \cite[Theorem~6.1]{K06-IWOTA} (see also \cite[Theorem~3.1]{K06}).
\begin{thm}\label{th:boundedness-PDO}
If $\fa\in C_b(\mR_+,V(\mR))$, then the Mellin pseudodifferential operator
$\operatorname{Op}(\fa)$, defined for functions $f\in C_0^\infty(\mR_+)$ by
the iterated integral
\[
[\operatorname{Op}(\fa)f](t)
=
\frac{1}{2\pi}\int\limits_\mR dx \int\limits_{\mR_+}
\fa(t,x)\left(\frac{t}{\tau}\right)^{ix}f(\tau) \frac{d\tau}{\tau}
\quad\mbox{for}\quad t\in\mR_+,
\]
extends to a bounded linear operator on the space $L^p(\mR_+,d\mu)$
and there is a number $C_p\in(0,\infty)$ depending only on $p$ such
that
\[
\|\operatorname{Op}(\fa)\|_{\cB(L^p(\mR_+,d\mu))}
\le C_p\|\fa\|_{C_b(\mR_+,V(\mR))}.
\]
\end{thm}
Obviously, if $\fa(t,x)=a(x)$ for all $(t,x)\in\mR_+\times\mR$, then the Mellin
pseudodifferential operator $\Op(\fa)$ becomes the Mellin convolution operator
\[
\Op(\fa)=\Co(a).
\]
\subsection{Algebra \boldmath{$\cE(\mR_+,V(\mR))$}}
Let $SO(\mR_+,V(\mR))$ denote the Banach subalgebra of $C_b(\mR_+,V(\mR))$
consisting of all $V(\mR)$-valued functions $\fa$ on $\mR_+$ that slowly
oscillate at $0$ and $\infty$, that is,
\[
\lim_{r\to 0} \operatorname{cm}_r^C(\fa)
=
\lim_{r\to \infty} \operatorname{cm}_r^C(\fa)=0,
\]
where
\[
\operatorname{cm}_r^C(\fa)
:=
\max\big\{
\big\|\fa(t,\cdot)-\fa(\tau,\cdot)\big\|_{L^\infty(\mR)}:t,\tau\in[r,2r]
\big\}.
\]

Let $\cE(\mR_+,V(\mR))$ be the Banach algebra of all $V(\mR)$-valued functions
$\fa\in SO(\mR_+,V(\mR))$ such that
\[
\lim_{|h|\to 0}\sup_{t\in\mR_+}\big\|\fa(t,\cdot)-\fa^h(t,\cdot)\big\|_V=0
\]
where $\fa^h(t,x):=\fa(t,x+h)$ for all $(t,x)\in\mR_+\times \mR$.

Let $\fa\in\cE(\mR_+,V(\mR))$. For every $t\in\mR_+$, the function $\fa(t,\cdot)$
belongs to $V(\mR)$ and, therefore, has finite limits at $\pm\infty$, which will
be denoted by $\fa(t,\pm\infty)$. Now we explain how to extend the function
$\fa$ to $\Delta\times\overline{\mR}$. By analogy with \cite[Lemma~2.7]{K06}
with the aid of Lemma~\ref{le:SO-fundamental-property}
one can prove the following.
\begin{lem}\label{le:values}
Let $s\in\{0,\infty\}$ and $\{\fa_k\}_{k=1}^\infty$ be a countable subset of the
algebra $\cE(\mR_+,V(\mR))$. For each $\xi\in M_s(SO(\mR_+))$ there is a sequence
$\{t_j\}_{j\in\mN}\subset\mR_+$ and functions $\fa_k(\xi,\cdot)\in V(\mR)$ such
that $t_j\to s$ as $j\to\infty$ and
\[
\fa_k(\xi,x)=\lim_{j\to\infty}\fa_k(t_j,x)
\]
for every $x\in\overline{\mR}$ and every $k\in\mN$.
\end{lem}
A straightforward application of Lemma~\ref{le:values} leads to the following.
\begin{lem}\label{le:values-sum-product}
Let $\fb\in\cE(\mR_+,V(\mR))$, $m,n\in\mN$, and $\fa_{ij}\in\cE(\mR_+,V(\mR))$ for
$i\in\{1,\dots,m\}$ and $j\in\{1,\dots,n\}$. If
\[
\fb(t,x)=\sum_{i=1}^m\prod_{j=1}^n\fa_{ij}(t,x),\quad (t,x)\in\mR_+\times\mR,
\]
then
\[
\fb(\xi,x)=\sum_{i=1}^m\prod_{j=1}^n\fa_{ij}(\xi,x),\quad (\xi,x)\in\Delta\times\overline{\mR}.
\]
\end{lem}
\begin{lem}[{\cite[Lemma 3.2]{KKL14b}}]
\label{le:values-series}
Let $\{\fa_n\}_{n\in\mathbb{N}}$ be a sequence of functions in $\cE(\mR_+,V(\mR))$
such that the series $\sum_{n=1}^\infty\fa_n$ converges in the norm of the algebra
$C_b(\mR_+,V(\mR))$ to a function $\fa\in\cE(\mR_+,V(\mR))$. Then
\begin{align}
&
\fa(t,\pm\infty) =\sum_{n=1}^\infty \fa_n(t,\pm\infty)
&& \mbox{for all}\quad
t\in\mR_+,
\label{eq:values-series-1}
\\
&
\fa(\xi,x) =\sum_{n=1}^\infty \fa_n(\xi,x)
&& \mbox{for all}\quad
(\xi,x)\in\Delta\times\mR.
\label{eq:values-series-2}
\end{align}
\end{lem}
\subsection{Products of Mellin Pseudodifferential Operators}
Applying the relation
\begin{equation}\label{eq:translation-PDO}
\Op(\fa)=E^{-1}a(x,D)E
\end{equation}
between the Mellin pseudodifferential operator $\Op(\fa)$ and the Fourier
pseudodifferential operator $a(x,D)$ considered in \cite{K06}, where
\begin{equation}\label{eq:translation-symbols}
\fa(t,x)=a(\ln t,x),\quad(t,x)\in\mR_+\times\mR,
\end{equation}
and $E$ is given by \eqref{eq:def-E}, we infer from
\cite[Theorem~8.3]{K06} the following compactness result.
\begin{thm}\label{th:comp-semi-commutators-PDO}
If $\fa,\fb\in\cE(\mR_+,V(\mR))$, then
\[
\operatorname{Op}(\fa)\operatorname{Op}(\fb)\simeq \operatorname{Op}(\fa\fb).
\]
\end{thm}
From
\eqref{eq:def-E}, \eqref{eq:translation-PDO}--\eqref{eq:translation-symbols},
\cite[Lemmas~7.1,~7.2]{K06}, and the proof of \cite[Lemma~8.1]{K06}
we can extract the following.
\begin{lem}\label{le:PDO-3-operators}
If $\fa,\fb,\fc\in\cE(\mR_+,V(\mR))$ are such that $\fa$ depends only on the
first variable and $\fc$ depends only on the second variable, then
\[
\operatorname{Op}(\fa)\operatorname{Op}(\fb)\operatorname{Op}(\fc)
=
\operatorname{Op}(\fa\fb\fc).
\]
\end{lem}
\subsection{Fredholmness of Mellin Pseudodifferential Operators}
To study the Fredholmness of Mellin pseudodifferential operators, we need the
Banach algebra $\widetilde{\cE}(\mR_+,V(\mR))$ consisting of all functions
$\fa$ belonging to $\cE(\mR_+,V(\mR))$ and such that
\[
\lim_{m\to\infty}\sup_{t\in\mR_+}\int\limits_{\mR\setminus[-m,m]}
\left|\frac{\partial \fa(t,x)}{\partial x}\right|\,dx=0.
\]

Now we are in a position to formulate the main result of this section.
\begin{thm}[{\cite[Theorem~5.8]{KKL14b}}]
\label{th:Fredholmness-PDO}
Suppose $\fa\in\widetilde{\cE}(\mR_+,V(\mR))$.
\begin{enumerate}
\item[{\rm(a)}]
If the Mellin pseudodifferential operator $\Op(\fa)$ is Fredholm on the space
$L^p(\mR_+,d\mu)$, then
\begin{equation}\label{eq:Fredholmness-PDO-1}
\fa(t,\pm\infty)\ne 0
\ \text{ for all }\
t\in\mR_+,
\quad
\fa(\xi,x)\ne 0
\ \text{ for all }\
(\xi,x)\in\Delta\times\overline{\mR}.
\end{equation}

\item[{\rm(b)}]
If \eqref{eq:Fredholmness-PDO-1} holds, then the Mellin pseudodifferential
operator $\Op(\fa)$ is Fredholm on the space $L^p(\mR_+,d\mu)$ and each its
regularizer has the form $\Op(\fb)+K$, where $K$ is a compact operator
on the space $L^p(\mR_+,d\mu)$ and $\fb\in\widetilde{\cE}(\mR_+,V(\mR))$
is such that
\begin{align*}
&\fb(t,\pm\infty)=1/\fa(t,\pm\infty)
\mbox{ for all }\ t\in\mR_+,
\\
&\fb(\xi,x)=1/\fa(\xi,x)
\mbox{ for all }\ (\xi,x)\in\Delta\times\overline{\mR}.
\end{align*}
\end{enumerate}
\end{thm}
Note that part (a) follows from \cite[Theorem~4.3]{K09} and part (b)
is the main result of \cite{KKL14b}.
\section{Applications of Mellin Pseudodifferential Operators}
\label{sec:applications-Mellin-PDO}
\subsection{Some Important Functions in the Algebra $\widetilde{\cE}(\mR_+,V(\mR))$}
\begin{lem}[{\cite[Lemma~4.2]{KKL14a}}]
\label{le:g-sp-rp}
Let $g\in SO(\mR_+)$. Then for every $y\in(1,\infty)$ the functions
\[
\fg(t,x):=g(t),
\quad
\fs_y(t,x):=s_y(x),
\quad
\fr_y(t,x):=r_y(x),
\quad
(t,x)\in\mR_+\times\mR,
\]
belong to the Banach algebra $\widetilde{\cE}(\mR_+,V(\mR))$.
\end{lem}
\begin{lem}[{\cite[Lemma~4.3]{KKL14a}}]
\label{le:fb}
Suppose $\omega\in SO(\mR_+)$ is a real-valued function. Then for every
$y\in(1,\infty)$ the function
\[
\fb(t,x):=e^{i\omega(t)x}r_y(x),
\quad (t,x)\in\mR_+\times\mR,
\]
belongs to the Banach algebra $\widetilde{\cE}(\mR_+,V(\mR))$ and
there is a positive constant $C(y)$ depending only on $y$ such that
\[
\|\fb\|_{C_b(\mR_+,V(\mR))}
\le
C(y)\left(1+\sup_{t\in\mR_+}|\omega(t)|\right).
\]
\end{lem}
\subsection{Operator $U_\gamma R_y$}
\begin{lem}[{\cite[Lemma~4.4]{KKL14a}}]
\label{le:shift-R-exact}
Let $\gamma\in SOS(\mR_+)$ and $U_\gamma$ be the associated isometric shift
operator on $L^p(\mR_+)$. For every $y\in(1,\infty)$, the operator $U_\gamma R_y$
can be realized as the Mellin pseudodifferential operator:
\[
U_\gamma R_y = \Phi^{-1}\Op (\fd) \Phi,
\]
where the function $\fd$, given for $(t,x)\in\mR_+\times\mR$ by
\[
\fd(t,x):=(1+t\psi'(t))^{1/p} e^{i\psi(t)x}r_y(x)
\quad\mbox{with}\quad
\psi(t):=\log[\gamma(t)/t],
\]
belongs to the algebra $\widetilde{\cE}(\mR_+,V(\mR))$.
\end{lem}
\subsection{Operator $(I-vU_\gamma)R_y$}
The previous lemma can be easily generalized to the case of operators containing
slowly oscillating coefficients.
\begin{lem}\label{le:A-R}
Let $y\in(1,\infty)$, $v\in SO(\mR_+)$, and $\gamma\in SOS(\mR_+)$.
Then
\begin{enumerate}
\item[{\rm(a)}]
the operator $(I-vU_\gamma) R_y$ can be realized as the Mellin pseudodifferential
operator:
\[
(I-vU_\gamma) R_y = \Phi^{-1}\Op (\fa) \Phi,
\]
where the function $\fa$, given for $(t,x)\in\mR_+\times\mR$ by
\[
\fa(t,x):=
\big(1-v(t)\big(\Psi(t)\big)^{1/p}e^{i\psi(t)x}\big)r_y(x)
\]
with $\psi(t):=\log[\gamma(t)/t]$ and $\Psi(t):=1+t\psi'(t)$,
belongs to $\widetilde{\cE}(\mR_+,V(\mR))$;

\item[{\rm(b)}]
we have
\[
\fa(\xi,x)=
\left\{\begin{array}{lll}
(1-v(\xi)e^{i\psi(\xi)x})r_y(x), &\mbox{if}& (\xi,x)\in\Delta\times\mR,
\\
0, &\mbox{if}& (\xi,x)\in(\mR_+\cup\Delta)\times\{\pm\infty\}.
\end{array}\right.
\]
\end{enumerate}
\end{lem}
\begin{proof}
(a) This statement follows straightforwardly from Lemmas~\ref{le:g-sp-rp},
\ref{le:shift-R-exact}, and \ref{le:PDO-3-operators}.

(b) If $t\in\mR_+$, then obviously
\begin{equation}\label{eq:A-R-1}
\fa(t,x)=0
\quad\mbox{for}\quad x\in\{\pm\infty\}.
\end{equation}
By Lemma~\ref{le:exponent-shift}, $\psi\in SO(\mR_+)$. Since $v,\psi\in SO(\mR_+)$,
from Lemma~\ref{le:g-sp-rp} it follows that the functions
\begin{equation}\label{eq:A-R-2}
\fv(t,x):=v(t),
\quad
\widetilde{\psi}(t,x):=\psi(t),
\quad(t,x)\in\mR_+\times\mR,
\end{equation}
belong to $\widetilde{\cE}(\mR_+,V(\mR))$. Consider the finite family
$\{\fa,\fv,\widetilde{\psi}\}\in\widetilde{\cE}(\mR_+,V(\mR))$. Fix $s\in\{0,\infty\}$
and $\xi\in M_s(SO(\mR_+))$. By Lemma~\ref{le:values} and \eqref{eq:A-R-2}, there
is a sequence $\{t_j\}_{j\in\mN}\subset\mR_+$ and a function $\fa(\xi,\cdot)\in V(\mR_+)$
such that
\begin{equation}\label{eq:A-R-3}
\lim_{j\to\infty}t_j=s,
\quad
v(\xi)=\lim_{j\to\infty}v(t_j),
\quad
\psi(\xi)=\lim_{j\to\infty}\psi(t_j),
\end{equation}
\begin{equation}\label{eq:A-R-4}
\fa(\xi,x)=\lim_{j\to\infty}\fa(t_j,x),
\quad x\in\overline{\mR}.
\end{equation}
From Lemmas~\ref{le:SO-nec} and~\ref{le:exponent-shift}  we obtain
\begin{equation}\label{eq:A-R-5}
\lim_{j\to\infty}(\Psi(t_j))^{1/p}=1.
\end{equation}
From \eqref{eq:A-R-1} and \eqref{eq:A-R-4} we get
\[
\fa(\xi,x)=0
\quad\mbox{for}\quad (\xi,x)\in(\mR_+\cup\Delta)\times\{\pm\infty\}.
\]
Finally, from \eqref{eq:A-R-3}--\eqref{eq:A-R-5} we obtain for $(\xi,x)\in\Delta\times\mR$,
\begin{align*}
\fa(\xi,x)
&=\lim_{j\to\infty}\fa(t_j,x)
\\
&=\left(
1-\left(\lim_{j\to\infty}v(t_j)\right)
\left(\lim_{j\to\infty}(\Psi(t_j))^{1/p}\right)
\exp\left(ix\lim_{j\to\infty}\psi(t_j)\right)
\right)r_y(x)
\\
&=
(1-v(\xi)e^{i\psi(\xi)x})r_y(x),
\end{align*}
which completes the proof.
\end{proof}
\subsection{Operator $(I-vU_\gamma)^{-1}R_y$}
The following statement is crucial for our analysis. It says that the operators
$(I-vU_\gamma)R_y$ and  $(I-vU_\gamma)^{-1}R_y$ have similar nature.
\begin{lem}\label{le:A-inverse-R}
Let $y\in(1,\infty)$, $v\in SO(\mR_+)$, and $\gamma\in SOS(\mR_+)$. If $1\gg v$,
then
\begin{enumerate}
\item[{\rm(a)}]
the operator $A:=I-vU_\gamma$ is invertible on $L^p(\mR_+)$;

\item[{\rm(b)}]
the operator $A^{-1}R_y$ can be realized as the Mellin pseudodifferential
operator:
\begin{equation}\label{eq:A-inverse-R-1}
A^{-1}R_y=\Phi^{-1}\Op(\fc)\Phi,
\end{equation}
where the function $\fc$, given for $(t,x)\in\mR_+\times\mR$ by
\begin{align}\label{eq:A-inverse-R-2}
&&\quad
\fc(t,x):=r_y(x)+\sum_{n=1}^\infty
\left(\prod_{k=0}^{n-1} v[\gamma_k(t)]\big(\Psi[\gamma_k(t)]\big)^{1/p}e^{i\psi[\gamma_k(t)]x}\right)
r_y(x)
\end{align}
with $\psi(t):=\log[\gamma(t)/t]$ and $\Psi(t):=1+t\psi'(t)$,
belongs to $\widetilde{\cE}(\mR_+,V(\mR))$;

\item[{\rm(c)}]
we have
\begin{align*}
&&\
\fc(\xi,x)=\left\{\begin{array}{lll}
(1-v(\xi)e^{i\psi(\xi)x})^{-1}r_y(x), &\mbox{if}& (\xi,x)\in\Delta\times\mR,
\\
0, &\mbox{if}& (\xi,x)\in(\mR_+\cup\Delta)\times\{\pm\infty\}.
\end{array}\right.
\end{align*}
\end{enumerate}
\end{lem}
\begin{proof}
(a) Since $1\gg v$, from Lemma~\ref{le:FO} we conclude that $A$ is invertible
on the space $L^p(\mR_+)$ and
\begin{equation}\label{eq:A-inverse-R-3}
A^{-1}=\sum_{n=0}^\infty (vU_\gamma)^n.
\end{equation}
Part (a) is proved.

(b) By Lemmas~\ref{le:alg-A} and~\ref{le:g-sp-rp},
\begin{equation}\label{eq:A-inverse-R-4}
R_y=\Phi^{-1}\Op(\fc_0)\Phi,
\end{equation}
where the function $\fc_0$, given by
\begin{equation}\label{eq:A-inverse-R-5}
\fc_0(t,x):=r_y(x),
\quad
(t,x)\in\mR_+\times\mR,
\end{equation}
belongs to $\widetilde{\cE}(\mR_+,V(\mR))$.

If $\gamma\in SOS(\mR_+)$, then from Lemma~\ref{le:iterations} it follows that
$\gamma_n\in SOS(\mR_+)$ for every $n\in\mZ$. By Lemma~\ref{le:exponent-shift}, the
functions
\begin{equation}\label{eq:A-inverse-R-6}
\psi_n(t):=\log\frac{\gamma_n(t)}{t},
\quad
\Psi_n(t):=1+t\psi_n'(t)
\quad t\in\mR_+,\quad n\in\mZ,
\end{equation}
are real-valued functions in $SO(\mR_+)\cap C^1(\mR_+)$.
For every $n\in\mN$,
\begin{equation}\label{eq:A-inverse-R-7}
(vU_\gamma)^nR_y=\left(\prod_{k=0}^{n-1}v\circ\gamma_k\right)U_{\gamma_n}R_y.
\end{equation}
By Lemma~\ref{le:shift-R-exact},
\begin{equation}\label{eq:A-inverse-R-8}
U_{\gamma_n}R_y=\Phi^{-1}\Op(\fd_n)\Phi,
\end{equation}
where the function $\fd_n$, given by
\begin{equation}\label{eq:A-inverse-R-9}
\fd_n(t,x):=\big(\Psi_n(t)\big)^{1/p}e^{i\psi_n(t)x}r_y(x),
\quad
(t,x)\in\mR_+\times\mR,
\end{equation}
belongs to $\widetilde{\cE}(\mR_+,V(\mR))$.
From \eqref{eq:A-inverse-R-6} it follows that
\[
\psi_n(t)
=
\log\frac{\gamma_{n-1}[\gamma(t)]}{t}
=
\log\frac{\gamma_{n-1}[\gamma(t)]}{\gamma(t)}+\log\frac{\gamma(t)}{t}
=
\psi_{n-1}[\gamma(t)]+\psi(t).
\]
Therefore
\begin{equation}\label{eq:A-inverse-R-10}
\psi_n'(t)=\psi_{n-1}'[\gamma(t)]\gamma'(t)+\psi'(t).
\end{equation}
By using $\gamma(t)=te^{\psi(t)}$ and $\gamma'(t)=\Psi(t)e^{\psi(t)}$, from
\eqref{eq:A-inverse-R-6} and \eqref{eq:A-inverse-R-10} we get
\begin{align*}
\Psi_n(t)
&=
t\psi_{n-1}'[\gamma(t)]\Psi(t)e^{\psi(t)}+(1+t\psi'(t))
\\
&=
\Psi(t)\big(1+\gamma(t)\psi_{n-1}'[\gamma(t)]\big)
=
\Psi(t)\Psi_{n-1}[\gamma(t)].
\end{align*}
From this identity by induction we get
\begin{equation}\label{eq:A-inverse-R-11}
\Psi_n(t)=\prod_{k=0}^{n-1}\Psi[\gamma_k(t)],
\quad
t\in\mR_+,
\quad
n\in\mN.
\end{equation}
From \eqref{eq:A-inverse-R-7}--\eqref{eq:A-inverse-R-9} and
\eqref{eq:A-inverse-R-11} we get
\begin{equation}\label{eq:A-inverse-R-12}
(vU_\gamma)^nR_y=\Phi^{-1}\Op(\fc_n)\Phi,
\quad
n\in\mN,
\end{equation}
where the function $\fc_n$ is given for $(t,x)\in\mR_+\times\mR$ by
\begin{equation}\label{eq:A-inverse-R-13}
\fc_n(t,x):=a_n(t)\fb_n(t,x)
\end{equation}
with
\begin{equation}\label{eq:A-inverse-R-14}
a_n(t):=\prod_{k=0}^{n-1}v[\gamma_k(t)]\big(\Psi[\gamma_k(t)]\big)^{1/p},
\quad
\fb_n(t,x):=e^{i\psi_n(t)x}r_y(x).
\end{equation}

By the hypothesis, $v\in SO(\mR_+)$. On the other hand, $\Psi\in SO(\mR_+)$
in view of Lemma~\ref{le:exponent-shift}. Hence $\Psi^{1/p}\in SO(\mR_+)$.
Then, due to Lemmas~\ref{le:composition} and \ref{le:iterations}, $a_n\in SO(\mR_+)$ for all $n\in\mN$.
Therefore, from Lemma~\ref{le:g-sp-rp} we obtain that $\fa_n(t,x):=a_n(t)$,
$(t,x)\in\mR_+\times\mR$, belongs to $\widetilde{\cE}(\mR_+,V(\mR))$.
On the other hand, by Lemma~\ref{le:fb}, $\fb_n\in\widetilde{\cE}(\mR_+,V(\mR))$.
Thus, $\fc_n=a_n\fb_n$ belongs to $\widetilde{\cE}(\mR_+,V(\mR))$ for every
$n\in\mN$.

Following the proof of \cite[Lemma~2.1]{KKL03} (see also \cite[Theorem~2.2]{A96}),
let us show that
\begin{equation}\label{eq:A-inverse-R-15}
\limsup_{n\to\infty}\|a_n\|_{C_b(\mR_+)}^{1/n}<1.
\end{equation}
By Lemmas~\ref{le:SO-nec} and \ref{le:exponent-shift},
\begin{equation}\label{eq:A-inverse-R-16}
\lim_{t\to s}\Psi(t)=1+\lim_{t\to s}t\psi'(t)=1,
\quad
s\in\{0,\infty\}.
\end{equation}
If $1\gg v$, then
\begin{equation}\label{eq:A-inverse-R-17}
\limsup_{t\to s}|v(t)|<1,
\quad
s\in\{0,\infty\}.
\end{equation}
From \eqref{eq:A-inverse-R-16}--\eqref{eq:A-inverse-R-17} it follows that
\[
L^*(s):=\limsup_{t\to s}\big|v(t)\big(\Psi(t)\big)^{1/p}\big|<1,
\quad
s\in\{0,\infty\}.
\]
Fix $\eps>0$ such that $L^*(s)+\eps<1$ for $s\in\{0,\infty\}$. By the definition
of $L^*(s)$, there exist points $t_1,t_2\in\mR_+$ such that
\begin{equation}\label{eq:A-inverse-R-18}
\begin{array}{lll}
\big|v(t)\big(\Psi(t)\big)^{1/p}\big|<L^*(0)+\eps
&\quad\mbox{for}\quad &t\in(0,t_1),
\\
\big|v(t)\big(\Psi(t)\big)^{1/p}\big|<L^*(\infty)+\eps
&\quad\mbox{for}\quad &t\in(t_2,\infty).
\end{array}
\end{equation}
The mapping $\gamma$ has no fixed points other than $0$ and $\infty$. Hence,
either $\gamma(t)>t$ or $\gamma(t)<t$ for all $t\in\mR_+$. For definiteness,
suppose that $\gamma(t)>t$ for all $t\in\mR_+$. Then there exists a number
$k_0\in\mN$ such that $\gamma_{k_0}(t_1)\in (t_2,\infty)$. Put
\[
M_1:=\sup_{t\in\mR_+}\big|v(t)\big(\Psi(t)\big)^{1/p}\big|,
\quad
M_2:=\sup_{t\in\mR_+\setminus[t_1,\gamma_{k_0}(t_1)]}\big|v(t)\big(\Psi(t)\big)^{1/p}\big|.
\]
Since $v\Psi^{1/p}\in SO(\mR_+)$, we have $M_1<\infty$. Moreover, from
\eqref{eq:A-inverse-R-18} we obtain
\[
M_2\le\max(L^*(0),L^*(\infty))+\eps<1.
\]
Then, for every $t\in\mR_+$ and $n\in\mN$,
\begin{align*}
|a_n(t)|
&=
\prod_{k=0}^{n-1}\big|v[\gamma_k(t)]\big(\Psi[\gamma_k(t)]\big)^{1/p}|
\\
&\le
M_1^{k_0}M_2^{n-k_0}
\le
M_1^{k_0}(\max(L^*(0),L^*(\infty))+\eps)^{n-k_0}.
\end{align*}
From here we immediately get \eqref{eq:A-inverse-R-15}.

Now let us show that
\begin{equation}\label{eq:A-inverse-R-19}
\limsup_{n\to\infty}\|\fb_n\|_{C_b(\mR_+,V(\mR))}^{1/n}\le 1.
\end{equation}
By Lemma~\ref{le:fb}, there exists a constant $C(y)\in(0,\infty)$ depending
only on $y$ such that for all $n\in\mN$,
\begin{equation}\label{eq:A-inverse-R-20}
\|\fb_n\|_{C_b(\mR_+,V(\mR))}
\le
C(y)\left(1+\sup_{t\in\mR_+}|\psi_n(t)|\right).
\end{equation}
From \eqref{eq:A-inverse-R-6} we obtain
\begin{equation}\label{eq:A-inverse-R-21}
\psi_n(t)
=\log\left(
\prod_{k=0}^{n-1}\frac{\gamma[\gamma_k(t)]}{\gamma_k(t)}\right)
=
\sum_{k=0}^{n-1}\log\frac{\gamma[\gamma_k(t)]}{\gamma_k(t)}
=
\sum_{k=0}^{n-1}\psi[\gamma_k(t)].
\end{equation}
Let
\[
M_3:=\sup_{t\in\mR_+}|\psi(t)|.
\]
Since $\gamma_k$ is a diffeomorphism of $\mR_+$ onto itself for every $k\in\mZ$,
we have
\begin{equation}\label{eq:A-inverse-R-22}
M_3
=
\sup_{t\in\mR_+}|\psi(t)|=\sup_{t\in\mR_+}|\psi[\gamma(t)]|=\dots=\sup_{t\in\mR_+}|\psi[\gamma_{n-1}(t)]|.
\end{equation}
From \eqref{eq:A-inverse-R-20}--\eqref{eq:A-inverse-R-22} we obtain
\[
\|\fb_n\|_{C_b(\mR_+,V(\mR))}
\le
C(y)(1+M_3n),
\quad n\in\mN,
\]
which implies \eqref{eq:A-inverse-R-19}. Combining \eqref{eq:A-inverse-R-13},
\eqref{eq:A-inverse-R-15}, and \eqref{eq:A-inverse-R-19}, we arrive at
\begin{align*}
\limsup_{n\to\infty}\|\fc_n\|_{C_b(\mR_+,V(\mR))}^{1/n}
\le&
\left(\limsup_{n\to\infty}\|a_n\|_{C_b(\mR_+)}^{1/n}\right)
\\
&\times
\left(\limsup_{n\to\infty}\|\fb_n\|_{C_b(\mR_+,V(\mR))}^{1/n}\right)<1.
\end{align*}
This shows that the series $\sum_{n=0}^\infty \fc_n$ is absolutely convergent
in the norm of $C_b(\mR_+,V(\mR))$. From \eqref{eq:A-inverse-R-13}--\eqref{eq:A-inverse-R-14}
and \eqref{eq:A-inverse-R-21} we get for  $(t,x)\in\mR_+\times\mR$ and $n\in\mN$,
\begin{equation}\label{eq:A-inverse-R-23}
\fc_n(t,x)=
\left(\prod_{k=0}^{n-1}v[\gamma_k(t)]\big(\Psi[\gamma_k(t)]\big)^{1/p}
e^{i\psi[\gamma_k(t)]x}\right)r_y(x).
\end{equation}
We have already shown that $\fc_0$ given by \eqref{eq:A-inverse-R-5} and $\fc_n$,
$n\in\mN$, given by \eqref{eq:A-inverse-R-23} belong to $\widetilde{\cE}(\mR_+,V(\mR))$.
Thus  $\fc:=\sum_{n=0}^\infty\fc_n$ is given by \eqref{eq:A-inverse-R-2} and it belongs
to $\widetilde{\cE}(\mR_+,V(\mR))$.

From \eqref{eq:A-inverse-R-4}, \eqref{eq:A-inverse-R-12} and Theorem~\ref{th:boundedness-PDO}
we get
\begin{align*}
\left\|\Phi^{-1}\Op(\fc)\Phi-\sum_{n=0}^N(vU_\gamma)^nR_y\right\|_{\cB(L^p(\mR_+))}
&=\left\|\Phi^{-1}\left(\fc-\sum_{n=0}^N\fc_n\right)\Phi\right\|_{\cB(L^p(\mR_+))}
\\
&\le
C_p\left\|\fc-\sum_{n=0}^N\fc_n\right\|_{C_b(\mR_+,V(\mR))}
\\
&=
o(1)
\quad \mbox{as}\quad N\to\infty.
\end{align*}
Hence
\[
\sum_{n=0}^\infty (vU_\gamma)^nR_y=\Phi^{-1}\Op(\fc)\Phi.
\]
Combining this identity with \eqref{eq:A-inverse-R-3}, we arrive at
\eqref{eq:A-inverse-R-1}. Part (b) is proved.

(c)
From \eqref{eq:A-inverse-R-5} and \eqref{eq:A-inverse-R-23} it follows that
$\fc_n(t,\pm\infty)=0$ for $n\in\mN\cup\{0\}$ and $t\in\mR_+$. Then, in
view of Lemma~\ref{le:values-series},
\begin{equation}\label{eq:A-inverse-R-24}
\fc(t,\pm\infty)=0,
\quad
t\in\mR_+.
\end{equation}

Since $v,\psi\in SO(\mR_+)$, from Lemma~\ref{le:g-sp-rp} it follows that
the functions
\begin{equation}\label{eq:A-inverse-R-25}
\mathfrak{v}(t,x)=v(t),
\quad
\widetilde{\psi}(t,x):=\psi(t),
\quad
(t,x)\in\mR_+\times\mR,
\end{equation}
belong to the Banach algebra $\widetilde{\cE}(\mR_+,V(\mR))$. Consider the countable
family $\{\mathfrak{v},\widetilde{\psi},\fc\}\cup\{\fc_n\}_{n=0}^\infty$
of functions in $\widetilde{\cE}(\mR_+,V(\mR))$.

Fix $s\in\{0,\infty\}$ and $\xi\in M_s(SO(\mR_+))$. By Lemma~\ref{le:values}
and \eqref{eq:A-inverse-R-25}, there is a sequence $\{t_j\}_{j\in\mN}\subset\mR_+$
and functions $\fc(\xi,\cdot)\in V(\mR_+)$, $\fc_n(\xi,\cdot)\in V(\mR_+)$,
$n\in\mN\cup\{0\}$, such that
\begin{equation}\label{eq:A-inverse-R-26}
\lim_{j\to\infty}t_j=s,
\quad
v(\xi)=\lim_{j\to\infty}v(t_j),
\quad
\psi(\xi)=\lim_{j\to\infty}\psi(t_j),
\end{equation}
and for $n\in\mN\cup\{0\}$ and $x\in\overline{\mR}$,
\begin{equation}\label{eq:A-inverse-R-27}
\fc_n(\xi,x)=\lim_{j\to\infty}\fc_n(t_j,x),
\quad
\fc(\xi,x)=\lim_{j\to\infty}\fc(t_j,x).
\end{equation}
From \eqref{eq:A-inverse-R-24} and the second limit in \eqref{eq:A-inverse-R-27} we get
\begin{equation}\label{eq:A-inverse-R-28}
\fc(\xi,\pm\infty)=\lim_{j\to\infty}\fc(t_j,\pm\infty)=0.
\end{equation}
Trivially,
\begin{equation}\label{eq:A-inverse-R-29}
\fc_0(\xi,x)=r_y(x),
\quad
(\xi,x)\in(\Delta\cup\mR_+)\times\overline{\mR}.
\end{equation}
From Lemmas~\ref{le:SO-nec} and \ref{le:exponent-shift} we obtain
\begin{equation}\label{eq:A-inverse-R-30}
\lim_{t\to s}\big(\Psi(t)\big)^{1/p}=1,
\quad
s\in\{0,\infty\}.
\end{equation}
From Lemma~\ref{le:iterations} it follows that for $k\in\mN$,
\begin{equation}\label{eq:A-inverse-R-31}
\lim_{j\to\infty}v(t_j)=\lim_{j\to\infty}v[\gamma_k(t_j)],
\quad
\lim_{j\to\infty}\psi(t_j)=\lim_{j\to\infty}\psi[\gamma_k(t_j)].
\end{equation}
Combining \eqref{eq:A-inverse-R-23}, \eqref{eq:A-inverse-R-26},
the first limit in \eqref{eq:A-inverse-R-27},
and \eqref{eq:A-inverse-R-30}--\eqref{eq:A-inverse-R-31}, we get for
$x\in\mR$ and $n\in\mN$,
\begin{align}
\fc_n(\xi,x)
&=
\lim_{j\to\infty}
\left(
\prod_{k=0}^{n-1}v[\gamma_k(t_j)]\big(\Psi[\gamma_k(t_j)]\big)^{1/p}e^{i\psi[\gamma_k(t_j)]x}
\right)r_y(x)
\nonumber
\\
&=
\big(v(\xi)e^{i\psi(\xi)x}\big)^nr_y(x).
\label{eq:A-inverse-R-32}
\end{align}
From \eqref{eq:A-inverse-R-29}, \eqref{eq:A-inverse-R-32}, and Lemma~\ref{le:values-series}  we obtain
\begin{equation}\label{eq:A-inverse-R-33}
\fc(\xi,x)
=
\sum_{n=0}^\infty \big(v(\xi)e^{i\psi(\xi)x}\big)^nr_y(x).
\end{equation}
Since $1\gg v$, we have
\[
\limsup_{t\to s}|v(t)|<1,
\quad
s\in\{0,\infty\},
\]
whence, in view of Lemma~\ref{le:SO-fundamental-property}, we obtain
\[
|v(\xi) e^{i\psi(\xi)x}|\le\max_{s\in\{0,\infty\}}\left(\limsup_{t\to s}|v(t)|\right)<1.
\]
Therefore,
\begin{equation}\label{eq:A-inverse-R-34}
\sum_{n=0}^\infty\big(v(\xi)e^{i\psi(\xi)x}\big)^n=\big(1-v(\xi)e^{i\psi(\xi)x}\big)^{-1}.
\end{equation}
From \eqref{eq:A-inverse-R-33} and \eqref{eq:A-inverse-R-34} we get
\begin{equation}\label{eq:A-inverse-R-35}
\fc(\xi,x)=\big(1-v(\xi)e^{i\psi(\xi)x}\big)^{-1}r_y(x),
\quad
(\xi,x)\in\Delta\times\mR.
\end{equation}
Combining \eqref{eq:A-inverse-R-24}, \eqref{eq:A-inverse-R-28}, and
\eqref{eq:A-inverse-R-35}, we arrive at the assertion of part (c).
\end{proof}
\section{Fredholmness and Index of the Operator $V$}\label{sec:proof}
\subsection{First Step of Regularization}
\begin{lem}\label{le:reg1}
Let $\alpha,\beta\in SOS(\mR_+)$ and let $c,d\in SO(\mR_+)$ be such that
$1\gg c$ and $1\gg d$. Then for every $\mu\in[0,1]$ and $y\in(1,\infty)$ the following
statements hold:
\begin{enumerate}
\item[{\rm (a)}]
the operators $I-\mu cU_\alpha$ and $I-\mu dU_\beta$ are invertible on $L^p(\mR_+)$;

\item[{\rm(b)}]
for $(t,x)\in\mR_+\times\mR$, the functions
\begin{align}
\fa_{\mu,y}^{c,\alpha}(t,x)
&:=(1-\mu c(t)(\Omega(t))^{1/p}e^{i\omega(t)x})r_y(x),
\label{eq:reg1-1}
\\
\fa_{\mu,y}^{d,\beta}(t,x)
&:=(1-\mu d(t)(H(t))^{1/p}e^{i\eta(t)x})r_y(x)
\label{eq:reg1-2}
\end{align}
and
\begin{align}
\fc_{\mu,y}^{c,\alpha}(t,x)
:=&
r_y(x)
\nonumber\\
&+\sum_{n=1}^\infty\mu^n
\left(\prod_{k=0}^{n-1}c[\alpha_k(t)]\big(\Omega[\alpha_k(t)]\big)^{1/p}
e^{i\omega[\alpha_k(t)]x}\right)r_y(x),
\label{eq:reg1-3}
\\
\fc_{\mu,y}^{d,\beta}(t,x)
:=&
r_y(x)
\nonumber\\
&+\sum_{n=1}^\infty\mu^n
\left(\prod_{k=0}^{n-1}d[\beta_k(t)]\big(H[\beta_k(t)]\big)^{1/p}
e^{i\eta[\beta_k(t)]x}\right)r_y(x),
\label{eq:reg1-4}
\end{align}
with
\begin{align}
&\omega(t)=\log[\alpha(t)/t],
\quad
\Omega(t)=1+t\omega'(t),
\label{eq:reg1-psi}
\\
&\eta(t)=\log[\beta(t)/t],
\quad
H(t)=1+t\eta'(t),
\label{eq:reg1-zeta}
\end{align}
belong to the algebra $\widetilde{\cE}(\mR_+,V(\mR))$;

\item[{\rm(c)}]
the operators
\begin{align}
V_{\mu,y}
&:=
(I-\mu cU_\alpha)P_y^++(I-\mu dU_\beta)P_y^-,
\label{eq:reg1-5}
\\
L_{\mu,y}
&:=
(I-\mu cU_\alpha)^{-1}P_y^++(I-\mu dU_\beta)^{-1}P_y^-
\label{eq:reg1-6}
\end{align}
are related by
\begin{equation}\label{eq:reg1-7}
V_{\mu,y}L_{\mu,y}\simeq L_{\mu,y}V_{\mu,y}\simeq H_{\mu,y},
\end{equation}
where
\begin{equation}\label{eq:reg1-8}
H_{\mu,y}:=\Phi^{-1}\Op(\fh_{\mu,y})\Phi
\end{equation}
and the function $\fh_{\mu,y}$, given for $(t,x)\in\mR_+\times\mR$ by
\begin{align}
\fh_{\mu,y}(t,x)
:=1+
\frac{1}{4}\big[&
2(r_y(x))^2
\nonumber\\
&-
\fa_{\mu,y}^{d,\beta}(t,x)\fc_{\mu,y}^{c,\alpha}(t,x)
-
\fa_{\mu,y}^{c,\alpha}(t,x)\fc_{\mu,y}^{d,\beta}(t,x)
\big],
\label{eq:reg1-9}
\end{align}
belongs to the algebra $\widetilde{\cE}(\mR_+,V(\mR))$.
\end{enumerate}
\end{lem}
\begin{proof}
(a) From Lemma~\ref{le:FO} it follows that the operators
\begin{equation}\label{eq:reg1-10}
I-\mu cU_\alpha,
I-\mu dU_\beta\in\mathcal{FO}_{\alpha,\beta}
\end{equation}
are invertible and
\begin{equation}\label{eq:reg1-11}
(I-\mu cU_\alpha)^{-1},
(I-\mu dU_\beta)^{-1}\in\mathcal{FO}_{\alpha,\beta}.
\end{equation}
This completes the proof of part (a).

(b) From Lemma~\ref{le:alg-A} it follows that
\begin{equation}\label{eq:reg1-12}
R_y^2 =\Phi^{-1}\Co(r_y^2)\Phi=\Phi^{-1}\Op(\fr_y^2)\Phi,
\end{equation}
where $\fr_y(t,x)=r_y(x)$ for $(t,x)\in\mR_+\times\mR$.
From Lemma~\ref{le:g-sp-rp} we deduce that $\fr_y^2\in\widetilde{\cE}(\mR_+,V(\mR))$.
By Lemma~\ref{le:A-R}(a),
\begin{align}
(I-\mu cU_\alpha)R_y
&=\Phi^{-1}\Op(\fa_{\mu,y}^{c,\alpha})\Phi,
\label{eq:reg1-13}
\\
(I-\mu dU_\beta)R_y
&=\Phi^{-1}\Op(\fa_{\mu,y}^{d,\beta})\Phi,
\label{eq:reg1-14}
\end{align}
where the functions $\fa_{\mu,y}^{c,\alpha}$ and $\fa_{\mu,y}^{d,\beta}$,
given by \eqref{eq:reg1-1} and \eqref{eq:reg1-2}, respectively, belong to
$\widetilde{\cE}(\mR_+,V(\mR))$. Lemma~\ref{le:A-inverse-R}(b) implies that
\begin{align}
(I-\mu cU_\alpha)^{-1}R_y
&=\Phi^{-1}\Op(\fc_{\mu,y}^{c,\alpha})\Phi,
\label{eq:reg1-15}
\\
(I-\mu dU_\beta)^{-1}R_y
&=\Phi^{-1}\Op(\fc_{\mu,y}^{d,\beta})\Phi,
\label{eq:reg1-16}
\end{align}
where the functions $\fc_{\mu,y}^{c,\alpha}$ and $\fc_{\mu,y}^{d,\beta}$ given by \eqref{eq:reg1-3}
and \eqref{eq:reg1-4}, respectively, belong to $\widetilde{\cE}(\mR_+,V(\mR))$.
In particular, this completes the proof of part (b).

(c) From \eqref{eq:reg1-10}--\eqref{eq:reg1-11} and Lemmas~\ref{le:compactness-commutators}
and \ref{le:alg-A} it follows that
\begin{align}
& (I-\mu cU_\alpha)^tT\simeq T(I-\mu cU_\alpha)^t,
\label{eq:reg1-17}
\\
& (I-\mu dU_\beta)^tT\simeq T(I-\mu dU_\beta)^t
\label{eq:reg1-18}
\end{align}
for every $t\in\{-1,1\}$ and $T\in\{P_y^+,P_y^-,R_y\}$. Applying consecutively
relations \eqref{eq:reg1-17}--\eqref{eq:reg1-18} with $T\in\{P_y^+,P_y^-\}$,
Lemma~\ref{le:PR-relations}(b), and relations \eqref{eq:reg1-17}--\eqref{eq:reg1-18}
with $T=R_y$,  we get
\begin{align}
V_{\mu,y}L_{\mu,y}
\simeq &
(P_y^+)^2+(I-\mu dU_\beta)(I-\mu cU_\alpha)^{-1}P_y^-P_y^+
\nonumber\\
&+(P_y^-)^2+(I-\mu cU_\alpha)(I-\mu dU_\beta)^{-1}P_y^+P_y^-
\nonumber\\
=&
\left(P_y^++\frac{R_y^2}{4}\right)-(I-\mu dU_\beta)(I-\mu cU_\alpha)^{-1}\frac{R_y^2}{4}
\nonumber\\
&+
\left(P_y^-+\frac{R_y^2}{4}\right)-(I-\mu cU_\alpha)(I-\mu dU_\beta)^{-1}\frac{R_y^2}{4}
\nonumber\\
\simeq &
I+\frac{R_y^2}{2}-\frac{1}{4}(I-\mu dU_\beta)R_y(I-\mu cU_\alpha)^{-1}R_y
\nonumber\\
&\quad\quad\quad\ -
\frac{1}{4}(I-\mu cU_\alpha)R_y(I-\mu dU_\beta)^{-1}R_y.
\label{eq:reg1-19}
\end{align}
Applying  equalities \eqref{eq:reg1-13}--\eqref{eq:reg1-16} to \eqref{eq:reg1-19}, we obtain
\begin{align}
V_{\mu,y}L_{\mu,y}
\simeq &
I+\frac{1}{2}\Phi^{-1}\Op(\fr_y^2)\Phi-\frac{1}{4}\Phi^{-1}\Op(\fa_{\mu,y}^{d,\beta})\Op(\fc_{\mu,y}^{c,\alpha})\Phi
\nonumber\\
&\quad\quad\quad\quad\quad\quad\quad\quad\
-\frac{1}{4}\Phi^{-1}\Op(\fa_{\mu,y}^{c,\alpha})\Op(\fc_{\mu,y}^{d,\beta})\Phi.
\label{eq:reg1-20}
\end{align}
From Theorem~\ref{th:comp-semi-commutators-PDO} we get
\begin{align}
\Op(\fa_{\mu,y}^{d,\beta})\Op(\fc_{\mu,y}^{c,\alpha})
&\simeq
\Op(\fa_{\mu,y}^{d,\beta}\fc_{\mu,y}^{c,\alpha}),
\label{eq:reg1-21}
\\
\Op(\fa_{\mu,y}^{c,\alpha})\Op(\fc_{\mu,y}^{d,\beta})
&\simeq
\Op(\fa_{\mu,y}^{c,\alpha}\fc_{\mu,y}^{d,\beta}).
\label{eq:reg1-22}
\end{align}
Combining \eqref{eq:reg1-20}--\eqref{eq:reg1-22}, we arrive at
\[
V_{\mu,y}L_{\mu,y}\simeq \Phi^{-1}\Op(\fh_{\mu,y})\Phi,
\]
where the function $\fh_{\mu,y}$, given by \eqref{eq:reg1-9}, belongs to the
algebra $\widetilde{\cE}(\mR_+,V(\mR))$ because the functions
\eqref{eq:reg1-1}--\eqref{eq:reg1-4} lie in this algebra in view of part (b).
Analogously, it can be shown that
\[
L_{\mu,y}V_{\mu,y}\simeq \Phi^{-1}\Op(\fh_{\mu,y})\Phi,
\]
which completes the proof.
\end{proof}
\subsection{Fredholmness of the Operator $H_{\mu,2}$}
In this subsection we will prove that the operators $H_{\mu,2}$ given by
\eqref{eq:reg1-8} are Fredholm for every $\mu\in[0,1]$. To this end, we will
use Theorem~\ref{th:Fredholmness-PDO}.

First we represent boundary values of $\fh_{\mu,y}$ in a way, which is convenient
for further analysis.
\begin{lem}\label{le:reg-factorization}
Let $\alpha,\beta\in SOS(\mR_+)$ and let $c,d\in SO(\mR_+)$ be such that
$1\gg c$ and $1\gg d$. If $\fh_{\mu,y}$ is given by \eqref{eq:reg1-9} and
\eqref{eq:reg1-1}--\eqref{eq:reg1-zeta}, then for every $\mu\in[0,1]$ and $y\in(1,\infty)$,
we have
\[
\fh_{\mu,y}(\xi,x)=\left\{\begin{array}{lll}
v_{\mu,y}(\xi,x)\ell_{\mu,y}(\xi,x), &\mbox{if}& (\xi,x)\in\Delta\times\mR,
\\
1, &\mbox{if}& (\xi,x)\in(\mR_+\cup\Delta)\times\{\pm\infty\},
\end{array}\right.
\]
where
\begin{align}
v_{\mu,y}(\xi,x) &:=(1-\mu c(\xi)e^{i\omega(\xi)x})p_y^+(x)+(1-\mu d(\xi)e^{i\eta(\xi)x})p_y^-(x),
\label{eq:reg-fact-2}
\\
\ell_{\mu,y}(\xi,x) &:=(1-\mu c(\xi)e^{i\omega(\xi)x})^{-1}p_y^+(x)+(1-\mu d(\xi)e^{i\eta(\xi)x})^{-1}p_y^-(x)
\label{eq:reg-fact-3}
\end{align}
for $(\xi,x)\in\Delta\times\mR$.
\end{lem}
\begin{proof}
From \eqref{eq:reg1-9}, Lemmas~\ref{le:values-sum-product}, \ref{le:A-R}(b), and
\ref{le:A-inverse-R}(c) it follows that
\[
\fh_{\mu,y}(\xi,x)=1
\quad\mbox{for}\quad
(\xi,x)\in(\mR_+\cup\Delta)\times\{\pm\infty\}
\]
and
\begin{align*}
\fh_{\mu,y}(\xi,x)
=&
1+\frac{1}{4}\big[2(r_y(x))^2
\\
&-
(1-\mu d(\xi)e^{i\eta(\xi)x})r_y(x)
(1-\mu c(\xi)e^{i\omega(\xi)x})^{-1}r_y(x)
\\
&-
(1-\mu c(\xi)e^{i\omega(\xi)x})r_y(x)
(1-\mu d(\xi)e^{i\eta(\xi)x})^{-1}r_y(x)\big]
\end{align*}
for $(\xi,x)\in\Delta\times\mR$. By Lemma~\ref{le:PR-relations}(a),
\begin{align*}
&\fh_{\mu,y}(\xi,x)
=
\\
=&
\left(p_y^+(x)+\frac{(r_y(x))^2}{4}\right)
-
(1-\mu d(\xi)e^{i\eta(\xi)x})
(1-\mu c(\xi)e^{i\omega(\xi)x})^{-1}\frac{(r_y(x))^2}{4}
\\
&+
\left(p_y^-(x)+\frac{(r_y(x))^2}{4}\right)
-
(1-\mu c(\xi)e^{i\omega(\xi)x})
(1-\mu d(\xi)e^{i\eta(\xi)x})^{-1}\frac{(r_y(x))^2}{4}
\\
=&
(p_y^+(x))^2
+
(1-\mu d(\xi)e^{i\eta(\xi)x})
(1-\mu c(\xi)e^{i\omega(\xi)x})^{-1}p_y^-(x)p_y^+(x)
\\
&+
(p_y^-(x))^2
+
(1-\mu c(\xi)e^{i\omega(\xi)x})
(1-\mu d(\xi)e^{i\eta(\xi)x})^{-1}p_y^+(x)p_y^-(x)
\\
=& v_{\mu,y}(\xi,x)\ell_{\mu,y}(\xi,x)
\end{align*}
for $(\xi,x)\in\Delta\times\mR$, which completes the proof.
\end{proof}
We were unable to prove that $\fh_{\mu,y}$ satisfies the hypotheses of
Theorem~\ref{th:Fredholmness-PDO} for every $y\in(1,\infty)$ or at least
for $y=p$. However, the very special form of the ranges of $v_{\mu,2}$ and
$\ell_{\mu,2}$ given by \eqref{eq:reg-fact-2} and \eqref{eq:reg-fact-3},
respectively, allows us to prove that $v_{\mu,2}$ and $\ell_{\mu,2}$ are
separated  from zero for all $\mu\in[0,1]$, and thus $\fh_{\mu,2}$ satisfies
the assumptions of Theorem~\ref{th:Fredholmness-PDO}.
\begin{lem}\label{le:Fredholmness-H}
Let $\alpha,\beta\in SOS(\mR_+)$ and let $c,d\in SO(\mR_+)$ be such that
$1\gg c$ and $1\gg d$. Then for every $\mu\in[0,1]$ the operator
$H_{\mu,2}$ given by \eqref{eq:reg1-8} is Fredholm on $L^p(\mR_+)$.
\end{lem}
\begin{proof}
By Lemma~\ref{le:reg-factorization}, for the function $\fh_{\mu,2}$ defined
by \eqref{eq:reg1-9} and \eqref{eq:reg1-1}--\eqref{eq:reg1-zeta} we have
\begin{equation}\label{eq:Fredholmness-H-1}
\fh_{\mu,2}(\xi,x)=1\ne 0
\quad\mbox{for}\quad
(\xi,x)\in(\mR_+\cup\Delta)\times\{\pm\infty\}
\end{equation}
and
\[
\fh_{\mu,2}(\xi,x)=v_{\mu,2}(\xi,x)\ell_{\mu,2}(\xi,x)
\quad\mbox{for}\quad
(\xi,x)\in\Delta\times\mR,
\]
where $v_{\mu,2}$ and $\ell_{\mu,2}$ are defined by \eqref{eq:reg-fact-2}
and \eqref{eq:reg-fact-3}, respectively. From Lemmas~\ref{le:range1} and
\ref{le:range2} it follows that for each $\xi\in\Delta$ the ranges of the
continuous functions $v_{\mu,2}(\xi,\cdot)$ and $\ell_{\mu,2}(\xi,\cdot)$
defined on $\mR$ lie in the half-plane
\[
\mathcal{H}^{\mu,\xi}:=\left\{
z\in\mC:\operatorname{Re}z> 1-\mu\max(|c(\xi)|,|d(\xi)|)
\right\}.
\]
From Lemma~\ref{le:SO-fundamental-property} we get
\begin{align*}
C(\Delta)&:=\sup_{\xi\in\Delta}|c(\xi)|=\max_{s\in\{0,\infty\}}\left(\limsup_{t\to s}|c(t)|\right),
\\
D(\Delta)&:=\sup_{\xi\in\Delta}|d(\xi)|=\max_{s\in\{0,\infty\}}\left(\limsup_{t\to s}|d(t)|\right).
\end{align*}
Since $1\gg c$ and $1\gg d$, we see that $C(\Delta)<1$ and $D(\Delta)<1$.
Therefore, for every $\xi\in\Delta$ and $\mu\in[0,1]$, the half-plane
$\mathcal{H}^{\mu,\xi}$ is contained in the half-plane
\[
\left\{z\in\mC:\operatorname{Re}z>
1-\max(|C(\Delta)|,|D(\Delta)|)
\right\}
\]
and the origin does not lie in the latter half-plane. Thus
\begin{equation}\label{eq:Fredholmness-H-2}
\fh_{\mu,2}(\xi,x)=v_{\mu,2}(\xi,x)\ell_{\mu,2}(\xi,x)\ne 0
\quad\mbox{for all}\quad
(\xi,x)\in\Delta\times\mR.
\end{equation}
From \eqref{eq:Fredholmness-H-1}--\eqref{eq:Fredholmness-H-2} and Theorem~\ref{th:Fredholmness-PDO}
we obtain that the operator $H_{\mu,2}$ is Fredholm on $L^p(\mR_+)$.
\end{proof}
\subsection{Proof of the Main Result}
For $\mu\in[0,1]$, consider the operators $V_{\mu,2}$ and $L_{\mu,2}$ defined by
\eqref{eq:reg1-5} and \eqref{eq:reg1-6}, respectively. It is obvious that
$V_{0,2}=P_y^++P_y^-=I$ and $V_{1,2}=V$. By Lemma~\ref{le:reg1}(c),
\begin{equation}\label{eq:proof-main-1}
V_{\mu,2}L_{\mu,2}\simeq L_{\mu,2}V_{\mu,2}\simeq H_{\mu,2},
\quad\mu\in[0,1],
\end{equation}
where the operator $H_{\mu,2}$ given by \eqref{eq:reg1-8} is Fredholm in view of
Lemma~\ref{le:Fredholmness-H}. Let $H_{\mu,2}^{(-1)}$ be a regularizer for $H_{\mu,2}$.
From \eqref{eq:proof-main-1} it follows that
\begin{equation}\label{eq:proof-main-2}
V_{\mu,2}(L_{\mu,2}H_{\mu,2}^{(-1)})\simeq I,
\quad
(H_{\mu,2}^{(-1)}L_{\mu,2})V_{\mu,2}\simeq I,
\quad \mu\in[0,1].
\end{equation}
Thus, $L_{\mu,2}H_{\mu,2}^{(-1)}$ is a right regularizer for $V_{\mu,2}$ and
$H_{\mu,2}^{(-1)}L_{\mu,2}$ is a left regularizer for $V_{\mu,2}$. Therefore,
$V_{\mu,2}$ is Fredholm for every $\mu\in[0,1]$. It is obvious that the
operator-valued function $\mu\mapsto V_{\mu,2}\in\cB(L^p(\mR_+))$ is continuous
on $[0,1]$. Hence the operators $V_{\mu,2}$ belong to the same connected component
of the set of all Fredholm operators. Therefore all $V_{\mu,2}$ have the same index
(see, e.g., \cite[Section~4.10]{GK92}). Since $V_{0,2}=I$, we conclude that
\[
\Ind V=\Ind V_{1,2}=\Ind V_{0,2}=\Ind I=0,
\]
which completes the proof of Theorem~\ref{th:main}.
\qed
\section{Regularization of the Operator $W$}
\label{sec:Regularization}
\subsection{Regularizers of the Operator $W$}
As a by-product of the proof of Section~\ref{sec:proof}, we can describe all regularizers
of a slightly more general operator $W$.
\begin{thm}\label{th:regularization-W}
Let $1<p<\infty$, $\eps_1,\eps_2\in\{-1,1\}$, and $\alpha,\beta\in SOS(\mR_+)$.
Suppose $c,d\in SO(\mR_+)$ are such that $1\gg c$ and $1\gg d$.
Then the operator $W$ given by
\[
W:=(I-cU_\alpha^{\eps_1})P_2^++(I-dU_\beta^{\eps_2})P_2^-
\]
is Fredholm on the space $L^p(\mR_+)$ and $\Ind W=0$. Moreover, each regularizer $W^{(-1)}$
of the operator $W$ is of the form
\begin{equation}\label{eq:regularization-W-1}
W^{(-1)}=[\Phi^{-1}\Op(\ff)\Phi]\cdot [(I- cU_\alpha^{\eps_1})^{-1}P_2^++(I-dU_\beta^{\eps_2})^{-1}P_2^-]+K,
\end{equation}
where $K\in\cK(L^p(\mR_+))$ and $\ff\in\widetilde{\cE}(\mR_+,V(\mR))$ is such that
\begin{equation}\label{eq:regularization-W-2}
\ff(\xi,x)=\left\{\begin{array}{lll}
\displaystyle\frac{1}{w(\xi,x)\ell(\xi,x)}, &\mbox{if}& (\xi,x)\in\Delta\times\mR,
\\[3mm]
1, &\mbox{if}& (\xi,x)\in(\mR_+\cup\Delta)\times\{\pm\infty\},
\end{array}\right.
\end{equation}
where
\begin{align}
w(\xi,x) &:=(1-c(\xi)e^{i\eps_1\omega(\xi)x})p_2^+(x)+(1-d(\xi)e^{i\eps_2\eta(\xi)x})p_2^-(x)\ne 0,
\label{eq:regularization-W-3}
\\
\ell(\xi,x) &:=\frac{p_2^+(x)}{1-c(\xi)e^{i\eps_1\omega(\xi)x}}+\frac{p_2^-(x)}{1-d(\xi)e^{i\eps_2\eta(\xi)x}}\ne 0
\label{eq:regularization-W-4}
\end{align}
for $(\xi,x)\in\Delta\times\mR$ with $\omega(t):=\log[\alpha(t)/t]$ and
$\eta(t):=\log[\beta(t)/t]$ for $t\in\mR_+$.
\end{thm}
\begin{proof}
Since $\alpha,\beta\in SOS(\mR_+)$, from Lemma~\ref{le:iterations} it follows that
$\alpha_{-1}$ and $\beta_{-1}$ also belong to $SOS(\mR_+)$. Taking into account that
$U_\alpha^{\eps_1}=U_{\alpha_{\eps_1}}$ and $U_\beta^{\eps_2}=U_{\beta_{\eps_2}}$,
from Theorem~\ref{th:main} we deduce that the operator $W$ is Fredholm and $\Ind W=0$.
Further, from \eqref{eq:proof-main-2} and Lemma~\ref{le:Fredholmness-H}
it follows that each regularizer $W^{(-1)}$ of $W$ is of the form
\begin{equation}\label{eq:regularization-W-5}
W^{(-1)}=H^{(-1)}L+K_1,
\end{equation}
where $K_1\in\cK(L^p(\mR_+))$,
\begin{equation}\label{eq:regularization-W-6}
L:=(I- cU_\alpha^{\eps_1})^{-1}P_2^++(I-dU_\beta^{\eps_2})^{-1}P_2^-,
\end{equation}
and $H^{(-1)}$ is a regularizer of the Fredholm operator $H$ given by
\[
H:=\Phi^{-1}\Op(\fh)\Phi,
\]
where $\fh\in\widetilde{\cE}(\mR_+,V(\mR))$ is given for $(t,x)\in\mR_+\times\mR$ by
\[
\fh(t,x):=1+
\frac{1}{4}\big[
2(r_2(x))^2
-
\fa_{1,2}^{d,\beta_{\eps_2}}(t,x)\fc_{1,2}^{c,\alpha_{\eps_1}}(t,x)
-
\fa_{1,2}^{c,\alpha_{\eps_1}}(t,x)\fc_{1,2}^{d,\beta_{\eps_2}}(t,x)
\big],
\]
and the functions
$\fa_{1,2}^{c,\alpha_{\eps_1}}$,
$\fc_{1,2}^{c,\alpha_{\eps_1}}$
and
$\fa_{1,2}^{d,\beta_{\eps_2}}$,
$\fc_{1,2}^{d,\beta_{\eps_2}}$
are given by \eqref{eq:reg1-1}--\eqref{eq:reg1-2} and \eqref{eq:reg1-3}--\eqref{eq:reg1-4}
with $\alpha$ and $\beta$ replaced by $\alpha_{\eps_1}$ and $\beta_{\eps_2}$, respectively.

Taking into account Lemma~\ref{le:inverse-shift-fibers}, by analogy with
Lemma~\ref{le:reg-factorization} we get
\begin{equation}\label{eq:regularization-W-7}
\fh(\xi,x)=\left\{\begin{array}{lll}
w(\xi,x)\ell(\xi,x), &\mbox{if}& (\xi,x)\in\Delta\times\mR,
\\
1, &\mbox{if}& (\xi,x)\in(\mR_+\cup\Delta)\times\{\pm\infty\}.
\end{array}\right.
\end{equation}
By Theorem~\ref{th:Fredholmness-PDO}(b), each regularizer $H^{(-1)}$ of the Fredholm
operator $H$ is of the form
\begin{equation}\label{eq:regularization-W-8}
H^{(-1)}=\Phi^{-1}\Op(\ff)\Phi+K_2,
\end{equation}
where $K_2\in\cK(L^p(\mR_+))$ and $\ff\in\widetilde{\cE}(\mR_+,V(\mR))$
is such that
\begin{equation}\label{eq:regularization-W-9}
\begin{array}{lll}
\ff(t,\pm\infty)=1/\fh(t,\pm\infty)
&\mbox{for all}& t\in\mR_+,
\\[3mm]
\ff(\xi,x)=1/\fh(\xi,x)
&\mbox{for all}& (\xi,x)\in\Delta\times\overline{\mR}.
\end{array}
\end{equation}
From \eqref{eq:regularization-W-5}--\eqref{eq:regularization-W-6} and \eqref{eq:regularization-W-8} we get
\eqref{eq:regularization-W-1}. Combining \eqref{eq:regularization-W-7}
and \eqref{eq:regularization-W-9}, we arrive at \eqref{eq:regularization-W-2}.
\end{proof}
\subsection{One Useful Consequence of Regularization of $W$}
\begin{thm}\label{th:for-index}
Under the assumptions of Theorem~{\rm\ref{th:regularization-W}}, for every
$y\in(1,\infty)$ there exists a function $\fg_y\in\widetilde{\cE}(\mR_+,V(\mR))$
such that
\begin{equation}\label{eq:for-index-1}
(\Phi^{-1}\Op(\fg_y)\Phi)W\simeq R_y
\end{equation}
and
\[
\fg_y(\xi,x)=\left\{\begin{array}{lll}
\displaystyle\frac{r_y(x)}{w(\xi,x)}, &\mbox{if}& (\xi,x)\in\Delta\times\mR,
\\[3mm]
0, &\mbox{if}& (\xi,x)\in(\mR_+\cup\Delta)\times\{\pm\infty\},
\end{array}\right.
\]
where the function $w$ is defined for $(\xi,x)\in\Delta\times\mR$ by \eqref{eq:regularization-W-3}.
\end{thm}
\begin{proof}
From Theorem~\ref{th:regularization-W} it follows that
\begin{equation}\label{eq:for-index-2}
(\Phi^{-1}\Op(\ff)\Phi)LWR_y\simeq R_y,
\end{equation}
where $L$ is given by \eqref{eq:regularization-W-6} and $\ff\in\widetilde{\cE}(\mR_+,V(\mR))$
satisfies \eqref{eq:regularization-W-2}. From Lemmas~\ref{le:compactness-commutators} and~\ref{le:alg-A}
we get
\begin{equation}\label{eq:for-index-3}
WR_y\simeq R_y W.
\end{equation}
Lemmas~\ref{le:alg-A} and~\ref{le:A-inverse-R}(a)--(b) imply that
\begin{align}
LR_y
&=(I-cU_\alpha^{\eps_1})^{-1}R_yP_2^++(I-dU_\beta^{\eps_2})^{-1}R_yP_2^-
\nonumber\\
&=
\Phi^{-1}\Op(\fc_{1,y}^{c,\alpha_{\eps_1}})\Co(p_2^+)\Phi
+
\Phi^{-1}\Op(\fc_{1,y}^{d,\beta_{\eps_2}})\Co(p_2^-)\Phi,
\label{eq:for-index-4}
\end{align}
where the functions $\fc_{1,y}^{c,\alpha_{\eps_1}}$ and $\fc_{1,y}^{d,\beta_{\eps_2}}$,
given by \eqref{eq:reg1-3} and \eqref{eq:reg1-4} with $\alpha$ and $\beta$ replaced by
$\alpha_{\eps_1}$ and $\beta_{\eps_2}$, respectively, belong to $\widetilde{\cE}(\mR_+,V(\mR))$.
From \eqref{eq:for-index-4} and Lemmas~\ref{le:g-sp-rp} and \ref{le:PDO-3-operators} we obtain
\begin{equation}\label{eq:for-index-5}
LR_y=\Phi^{-1}\Op(\fc_{1,y}^{c,\alpha_{\eps_1}}p_2^++\fc_{1,y}^{d,\beta_{\eps_2}}p_2^-)\Phi,
\end{equation}
where $\fc_{1,y}^{c,\alpha_{\eps_1}}p_2^++\fc_{1,y}^{d,\beta_{\eps_2}}p_2^-\in\widetilde{\cE}(\mR_+,V(\mR))$.
From \eqref{eq:for-index-2}--\eqref{eq:for-index-3}, \eqref{eq:for-index-5}, and
Theorem~\ref{th:comp-semi-commutators-PDO} we get \eqref{eq:for-index-1} with
\begin{equation}\label{eq:for-index-6}
\fg_y:=\ff\,(\fc_{1,y}^{c,\alpha_{\eps_1}}p_2^++\fc_{1,y}^{d,\beta_{\eps_2}}p_2^-)\in\widetilde{\cE}(\mR_+,V(\mR)).
\end{equation}
Obviously,
\begin{equation}\label{eq:for-index-7}
p_2^\pm(\pm\infty)=1,
\quad
p_2^\pm(\mp\infty)=0.
\end{equation}
By Lemmas~\ref{le:inverse-shift-fibers} and \ref{le:A-inverse-R}(c),
\begin{align}
\fc_{1,y}^{c,\alpha_{\eps_1}}(\xi,x)
&=
\left\{\begin{array}{ll}
\displaystyle\frac{r_y(x)}{1-c(\xi)e^{i\eps_1\omega(\xi)x}},
&\mbox{if } (\xi,x)\in\Delta\times\mR,
\\[3mm]
0, &\mbox{if } (\xi,x)\in(\mR_+\cup\Delta)\times\{\pm\infty\},
\end{array}\right.
\label{eq:for-index-8}
\\
\fc_{1,y}^{d,\beta_{\eps_2}}(\xi,x)
&=
\left\{\begin{array}{ll}
\displaystyle\frac{r_y(x)}{1-d(\xi)e^{i\eps_2\eta(\xi)x}},
&\mbox{if } (\xi,x)\in\Delta\times\mR,
\\[3mm]
0, &\mbox{if } (\xi,x)\in(\mR_+\cup\Delta)\times\{\pm\infty\}.
\end{array}\right.
\label{eq:for-index-9}
\end{align}
From \eqref{eq:for-index-6}--\eqref{eq:for-index-9},
\eqref{eq:regularization-W-2}--\eqref{eq:regularization-W-4},
and Lemma~\ref{le:values-sum-product} we get
\[
\fg_y(\xi,x)=0
\quad\mbox{for}\quad (\xi,x)\in(\mR_+\cup\Delta)\times\{\pm\infty\}
\]
and
\begin{align*}
\fg_y(\xi,x)
&=
\ff(\xi,x)\left(
\frac{r_y(x)p_2^+(x)}{1-c(\xi)e^{i\eps_1\omega(\xi)x}}
+
\frac{r_y(x)p_2^-(x)}{1-d(\xi)e^{i\eps_2\eta(\xi)x}}
\right)
\\
&=
\frac{\ell(\xi,x)r_y(x)}{w(\xi,x)\ell(\xi,x)}
=
\frac{r_y(x)}{w(\xi,x)}
\end{align*}
for $(\xi,x)\in\Delta\times\mR$.
\end{proof}
Relation \eqref{eq:for-index-1} will play an important role in the proof
of an index formula for the operator $N$ in \cite{KKL15-progress}.


\begin{thebibliography}{99}
\bibitem{A96}
Antonevich, A.B.:
Linear Functional Equations. Operator Approach.
Operator Theory: Advances and Applications, vol. 83. Birkh\"auser, Basel (1996)

\bibitem{BS06}
B\"ottcher, A., Silbermann, B.:
Analysis of Toeplitz Operators. 2nd edn. Springer, Berlin (2006)

\bibitem{D87}
Duduchava, R.:
On algebras generated by convolutions and discontinuous
functions.
Integral Equat. Oper. Theory \textbf{10}, 505--530 (1987)

\bibitem{GK92}
Gohberg, I., Krupnik, N.:
One-Dimensional Linear Singular Integral Equations. I. Introduction.
Operator Theory: Advances and Applications, vol. 53. Birkh\"auser, Basel (1992)

\bibitem{HRS94}
Hagen, R., Roch, S., Silbermann, B.:
Spectral Theory of Approximation Methods for Convolution Equations.
Operator Theory: Advances and Applications, vol. 74. Birkh\"auser, Basel (1994)

\bibitem{KKL03}
Karlovich, A.Yu.,  Karlovich, Yu.I., Lebre, A.B.:
Invertibility of functional operators with slowly
oscillating non-Carleman shifts.
In: ``Singular Integral Operators, Factorization and
Applications".
Operator Theory: Advances and Applications, vol. 142,
pp. 147--174 (2003)

\bibitem{KKL11a}
Karlovich, A.Yu., Karlovich, Yu.I., Lebre, A.B.:
Sufficient conditions for Fredholmness of singular integral
operators with shifts and slowly oscillating data.
{Integr. Equ. Oper. Theory \textbf{70}, 451--483 (2011)}

\bibitem{KKL11b}
Karlovich, A.Yu., Karlovich, Yu.I., Lebre, A.B.:
Necessary conditions for Fredholmness of singular integral
operators with shifts and slowly oscillating data.
{Integr. Equ. Oper. Theory \textbf{71}, 29--53 (2011)}

\bibitem{KKL14a}
Karlovich, A.Yu., Karlovich, Yu.I., Lebre, A.B.:
Fredholmness and index of simplest singular integral operators with two slowly oscillating shifts.
Operators and Matrices, \textbf{8}, no. 4, 935--955 (2014)

\bibitem{KKL14b}
Karlovich, A.Yu., Karlovich, Yu.I., Lebre, A.B.:
On regularization of Mellin PDO's with slowly oscillating symbols of limited smoothness.
Comm. Math. Anal. \textbf{17}, no. 2, 189--208 (2014)

\bibitem{KKL15-progress}
Karlovich, A.Yu., Karlovich, Yu.I., Lebre, A.B.:
The index of weighted singular integral operators with shifts and slowly oscillating data.
Work in progress (2015)

\bibitem{K06}
Karlovich, Yu.I.:
An algebra of pseudodifferential operators with slowly oscillating symbols.
{Proc. London Math. Soc. \textbf{92}, 713--761 (2006)}

\bibitem{K06-IWOTA}
Karlovich, Yu.I.:
Pseudodifferential operators with compound slowly oscillating symbols.
In: ``The Extended Field of Operator Theory".
{Operator Theory: Advances and Applications, vol. 171,
pp. 189--224 (2006)}

\bibitem{K08}
Karlovich, Yu.I.:
Nonlocal singular integral operators with slowly oscillating data.
In: ``Operator Algebras, Operator Theory and Applications".
{Operator Theory: Advances and Applications, vol. 181, pp. 229--261 (2008)}

\bibitem{K09}
Karlovich, Yu.I.:
An algebra of shift-invariant singular integral operators with slowly
oscillating data and its application to operators with a Carleman shift.
In: ``Analysis, Partial Differential Equations and Applications.
The Vladimir Maz'ya Anniversary Volume".
{Operator Theory: Advances and Applications, vol. 193, pp. 81--95 (2009)}

\bibitem{R92}
Rabinovich, V.S.:
Singular integral operators on a composed contour with oscillating
tangent and pseudodifferential Mellin operators.
Soviet Math. Dokl. \textbf{44}, 791--796 (1992)

\bibitem{RRS04}
Rabinovich, V.S., Roch, S., Silbermann, B.:
Limit Operators and Their Applications in Operator Theory.
Operator Theory: Advances and Applications, vol. 150.
Birkh\"auser, Basel (2004)

\bibitem{RSS11}
Roch, S., Santos, P.A., Silbermann, B.:
Non-commutative Gelfand theories.
A tool-kit for operator theorists and numerical analysts.
{Universitext. Springer-Verlag London, London (2011)}

\bibitem{SM86}
Simonenko, I.B., Chin Ngok Minh:
Local method in the theory of one-dimensional singular integral equations with
piecewise continuous coefficients. Noetherity.
Rostov-on-Don State Univ., Rostov-on-Don, in Russian (1986)
\end{thebibliography}
\end{document}